\documentclass[12pt]{amsart}
\usepackage{times}
\usepackage{amssymb,amscd,amsthm, verbatim,amsmath,color,fancyhdr, mathrsfs}
\usepackage{tikz}
\usepackage{graphicx}
\usepackage{turnstile}
\usepackage{mathtools}
\usepackage{bm}

\usepackage{verbatim}
\usepackage[a4paper, left=2.75cm, right=2.75cm, top=2.75cm,bottom=3.75cm,dvips]{geometry}
\usepackage[hidelinks, colorlinks=true, linkcolor = blue!80!black, citecolor=blue!40!black]{hyperref}

\usepackage{pgfplots}
\pgfplotsset{compat=newest}
\usetikzlibrary{calc, intersections}

\newtheorem{thm}{Theorem}[section]
\newtheorem{prop}[thm]{Proposition}
\newtheorem{lem}[thm]{Lemma}

\newtheorem{cor}[thm]{Corollary}

\newtheorem{conj}[thm]{Conjecture}

\theoremstyle{definition}
\newtheorem{definition}[thm]{Definition}
\newtheorem{remark}[thm]{Remark}
\newtheorem{example}[thm]{Example}

\newcommand{\FF}{\mathcal{F}}

\title{Nerves, minors, and piercing numbers}

\author{Andreas~F.~Holmsen}
\address{A.~F.~Holmsen \\ Department of Mathematical Sciences \\ KAIST \\ Daejeon \\ South Korea} \email{andreash@kaist.edu}

\author{Minki~Kim}
\address{M.~Kim \\ Department of Mathematical Sciences \\ KAIST \\ Daejeon \\ South Korea} \email{kmk90@kaist.ac.kr}

\author{Seunghun~Lee}
\address{S.~Lee \\ Department of Mathematical Sciences \\ KAIST \\ Daejeon \\ South Korea} \email{prosolver@kaist.ac.kr}

\thanks{All authors were supported by Basic Science research Program through the National Research Foundation of Korea (NRF) funded by the Ministry of Education (NRF-2016R1D1A1B03930998). The first author was also partially supported by Swiss National Science Foundation grants 200020-165977 and 200021-162884. }

\date{\today}

\begin{document}

\begin{abstract}
We make the first step towards a ``nerve theorem'' for graphs. Let $G$ be a simple graph and let $\FF$ be a family of induced subgraphs of $G$ such that the intersection of any members of $\FF$ is either empty or connected. We show that if the nerve complex of $\FF$ has non-vanishing homology in dimension three, then $G$ contains the complete graph on five vertices as a minor. As a consequence we confirm a conjecture of Goaoc concerning an extension of the planar $(p,q)$ theorem due to Alon and Kleitman: Let $\FF$ be a finite family of open connected sets in the plane such that the intersection of any members of $\FF$ is either empty or connected. If among any $p \geq 3$ members of $\FF$ there are some three that intersect, then there is a set of $C$ points which intersects every member of $\FF$, where $C$ is a constant depending only on $p$.
\end{abstract}
\maketitle

\section{Introduction} \label{intros}

\subsection{Connected covers in graphs} \label{concovers}
Given a family of non-empty sets $\mathcal{F} = \{S_1, \dots, S_n\}$ we may associate with it an abstract simplicial complex on the vertex set $[n] = \{1, \dots, n\}$  defined as
\[N(\mathcal{F}) = \{ \sigma \subset [n] \; : \; \textstyle{\bigcap_{i\in \sigma}} S_i \neq \emptyset \}.\] 
The complex $N(\mathcal{F})$ is called the {\bf{\em nerve}} of the family $\mathcal{F}$.
Nerve complexes occur frequently in topological combinatorics, typically when $\mathcal{F} = \{K_1, \dots, K_n\}$ is a family of simplicial complexes. In this situation, an indispensable tool is a ``nerve theorem'' which allows us to relate the topology of $N(\mathcal{F})$ to that of $\bigcup_{i\in [n]}K_i$.

Informally speaking, a nerve theorem asserts that if $K_\sigma = \bigcap_{i\in \sigma}K_i$ is sufficiently connected (in terms of homotopy \cite{bjorner} or homology \cite{meshulam}) for every $\sigma \in N(\FF)$, then $N(\mathcal{F})$ adequately reflects the topology of $\bigcup_{i\in [n]}K_i$ (in terms of isomorphisms of certain homotopy or homology groups). A classical example is Borsuk's nerve theorem \cite{borsuk} which states that if $K_\sigma$ is contractible for every $\sigma\in N(\mathcal{F})$, then $N(\mathcal{F})$ is homotopy equivalent to $\bigcup_{i\in [n]}K_i$. (See \cite{xavi-adv,govc,hell, montejano} for recent variations of nerve theorems and \cite{bjorner-survey} for applications in combinatorics.)  

\medskip

The focus of this paper will be on the situation when $\mathcal{F} = \{G_1, \dots, G_n\}$ is a family of graphs. (All graphs considered here are finite, simple, and undirected.) By viewing  graphs as 1-dimensional simplicial complexes, we may apply one of the standard nerve theorems to obtain information about the homotopy type or the cycle space of the graph $G = \bigcup_{i\in [n]}G_i$, provided the intersection $G_\sigma = \bigcap_{i\in \sigma}G_i$ is a tree for every $\sigma\in N(\mathcal{F})$. Unfortunately, this condition is quite restrictive and the conclusion gives us rather limited structural information about the graph. 

Recall that a graph $H$ is a {\bf{\em minor}} of a graph $G$ if there exists pairwise disjoint connected subgraphs $G_v\subset G$, one for each vertex $v$ of $H$, such that for any pair of adjacent vertices $u$ and $v$ in $H$, there is an edge in $G$ connecting a vertex in $G_u$ to a vertex in $G_v$ \cite{mohar}. When $H$ is a minor of $G$ we denote this by $H\prec G$.

One of our main goals is to show that the homology of the nerve of $\mathcal{F}$ can reveal information about minors in the graph $G = \bigcup_{i\in [n]}G_i$. In order to make this work it is necessary to relax the condition that intersections are contractible, and instead require only that they are connected. This motivates one of the key concepts of this paper.

\begin{definition} \label{def:connected cover}
A {\bf{\em connected cover}} in a graph $G$ is a finite family $\mathcal{F} = \{G_1, \dots, G_n\}$ of induced subgraphs of $G$ such that 
 $G_\sigma = \bigcap_{i\in \sigma}G_i$ is connected for every $\sigma \in N(\mathcal{F})$.
\end{definition}

In general we should not expect that the nerve of an arbitrary connected cover will tell us much about the structure of the graph.
Indeed, one can obtain rather trivial simplicial complexes in this way. For instance, if $\mathcal{F}$ consists of either pairwise disjoint induced subgraphs of $G$, or of many copies of the same induced subgraph of $G$, then $N(\mathcal{F})$ is either a set of isolated vertices, or a simplex.

On the other hand, the structure of the graph may be reflected in the nerves of certain connected covers. To see this, suppose that $K_{d+2}$, the complete graph on $d+2$ vertices, is a minor of $G$ for some $d\geq 1$. We may assume that $G$ is connected, or else we just consider the connected component containing $K_{d+2}$ as a minor. Then there is a partition of the vertex set $V = V(G)$ into $d+2$ parts $V_1, \dots, V_{d+2}$ such that the induced subgraphs $G[V_i]$ are connected for every $i$, and such that for every $i\neq j$ there exists an edge of $G$ connecting a vertex in $V_i$ to a vertex in $V_j$. If we let $G_i = G[V\setminus V_i]$, then the family $\mathcal{F} = \{G_1, \dots, G_{d+2}\}$ is a connected cover in $G$. Furthermore, $N(\FF)$ is the boundary of the $(d+1)$-dimensional simplex. This shows that if $K_{d+2}\prec G$, then there exists a connected cover in $G$ whose nerve has non-vanishing homology in dimension $d$.

\subsection{The homological dimension of a graph}
For a simplicial complex $K$ let $\tilde{H}_i(K)$ denote the $i$-th reduced homology group of $K$ with coefficients in $\mathbb{Z}_2$. A simple measure of the ``complexity'' of a connected cover is the greatest dimension for which the homology of its nerve is non-vanishing. Taking the maximum over all connected covers in a graph gives us the following graph invariant. 

\begin{definition} 
The {\bf{\em homological dimension}} of a graph $G$, denoted by $\gamma(G)$, is the greatest integer $d$ such that $\tilde{H}_d(N(\mathcal{F})) \neq 0$ for some connected cover $\mathcal{F}$ in $G$. For the single vertex graph $K_1$ we define $\gamma(K_1) = -1$.
\end{definition}

Note that the homological dimension is well-defined for any finite graph $G$, and that $\gamma(G)\geq 0$ for any graph $G$ with at least two vertices. It is also not difficult to show that the homological dimension is {\em minor-monotone} in the sense that if $H\prec G$, then $\gamma(H)\leq \gamma(G)$. 

We argued above that if $K_{d+2}\prec G$, then $\gamma(G)\geq d$. One of our main results is that the converse holds for small values of $d$. 

\begin{thm} \label{main-thm}
For any graph $G$ the following hold.
\begin{enumerate}
\item\label{k3minor} $K_3\prec G\iff\gamma(G)\geq 1$.
\item\label{k4minor} $K_4\prec G\iff\gamma(G)\geq 2$.
\item\label{k5minor} $K_5\prec G\iff\gamma(G)\geq 3$.
\end{enumerate}
\end{thm}

Part \eqref{k3minor} of Theorem \ref{main-thm} is an easy consequence of Borsuk's nerve theorem and can be argued as follows. Suppose $K_3$ is not a minor of $G$. Then $G$ is a forest, and for any connected cover $\mathcal{F} = \{G_1, \dots, G_n\}$ in $G$ it follows that $G_\sigma$ is contractible for every $\sigma\in N(\mathcal{F})$. Therefore Borsuk's nerve theorem implies that $N(\mathcal{F})$ is homotopy equivalent to $\bigcup_{i\in[n]}G_i$ which is either contractible, or a disjoint union of contractible sets. This shows that $\tilde{H}_i(N(\FF)) = 0$ for all $i\geq 1$. 

\medskip

Parts \eqref{k4minor} and \eqref{k5minor} of Theorem \ref{main-thm} rely on well-known structure theorems for graphs without $K_4$ or $K_5$ minors, due to Wagner \cite{wagner} (see also e.g. \cite[chapter 7.3]{diestel}). These results describe how such graphs can be constructed by clique-sum operations. This is useful in our situation because if $G$ is a clique-sum of graphs $G_1$ and $G_2$, we can show that \[\gamma(G)  = \max\{\gamma(G_1), \gamma(G_2)\},\] 
by application of the Mayer--Vietoris exact sequence. The crucial step towards establishing part \eqref{k5minor} is Theorem \ref{planar}, which states that $\gamma(G)\leq 2$ for any planar graph $G$. This is our most technical result and the proof relies heavily on a certain combinatorial identity (Lemma \ref{isoperm}) which is where the choice of homology with $\mathbb{Z}_2$-coefficients plays the most crucial role. 

\medskip

The basic properties of the homological dimension outlined above will be stated precisely and proven in Section \ref{invariant}. They suggest a close relationship between the homological dimension and other minor-monotone graph invariants \cite{edmonds,vdh,schr}, in particular the invariant discussed in \cite[section 5]{vdh} and \cite[section 18]{schr}.
Our results show that $K_3$, $K_4$, and $K_5$ minors in a graph are perfectly detected by its homological dimension. Despite this limited evidence, we are tempted to conjecture that this holds for all complete minors. 

\begin{conj} \label{minor-nerve}
For every positive integer $d$ and graph $G$, 
\[K_{d+2}\prec G \iff \gamma(G)\geq d.\]
\end{conj}

\begin{remark}[Restriction to connected graphs] We will throughout restrict our attention to {\em connected graphs}. The main reason for this restriction is to avoid certain trivial exceptional cases in some of our statements and proofs. This restriction does not lead to any loss of generality, because if $\mathcal{F}$ is a connected cover of a disconnected graph $G$, then $N(\mathcal{F})$ is entirely determined by the restriction of $\FF$ to each individual component of $G$. As a consequence, for any graph $G$ with at least one edge, we have \[\gamma(G) = \max_H \gamma(H)\] where $H$ ranges over all connected components of $G$.
\end{remark}

\begin{remark}[The induced subgraph condition]
It is possible to relax the assumption that the members of a connected cover are {\em induced subgraphs}, but this does not lead to greater generality. To see this, consider a family $\mathcal{F} = \{G_1, \dots, G_n\}$ of subgraphs of $G$ (not necessarily induced) such that $G_\sigma$ is connected for every $\sigma \in N(\FF)$. If we define the family $\mathcal{F}' = \{G_1', \dots, G_n'\}$, where $G_i'$ is the subgraph induced by the vertices of $G_i$, it is obvious that $\mathcal{F}'$ is also a connected cover of $G$ and that the nerves $N(\mathcal{F})$ and $N(\mathcal{F}')$ are isomorphic.
\end{remark}

\medskip

\subsection{Piercing problems}
Another area in which nerve complexes play a prominent role is in the study of Helly-type theorems \cite{eckhoff, wenger} and the intersection patterns of convex sets \cite{tancer}. This follows from the fact that every non-empty convex set is contractible and that intersections of convex sets are convex. Therefore Borsuk's nerve theorem implies that the nerve of a family of convex sets is homotopy equivalent to the union of the members of the family. This viewpoint has uncovered many deep results which reach far beyond the setting of convexity \cite{xavi-adv,eck-upper,xbetti,upperbound,topcol,topamenta}. In fact, our notion of connected covers was motivated by such an application.

For positive integers $p\geq q$, a family of sets $\mathcal{F}$ has the $\bm{(p,q)}$ {\bf{\em property}} if among any $p$ members of $\mathcal{F}$ there are some $q$ that are intersecting, that is, $q$ members whose intersection is non-empty. The {\bf{\em piercing number}} of $\mathcal{F}$ is the minimal number $k$ such that $\mathcal{F}$ can be partitioned into $k$ intersecting subfamilies. As an application of connected covers we prove the following result conjectured by Xavier Goaoc (personal communication).

\begin{thm}\label{connected pq}
For any integers $p\geq q \geq 3$ there exists an integer $C = C(p,q)$ such that the following holds. Let $\mathcal{F}$ be a finite family of open connected sets in the plane satisfying the condition that the intersection of any  members of $\mathcal{F}$ is empty or connected. If $\mathcal{F}$ has the $(p,q)$ property, then the piercing number of $\mathcal{F}$ is at most $C$.
\end{thm}

In the special case when the members of $\mathcal{F}$ are convex sets, Theorem \ref{connected pq} reduces to the planar version of the celebrated $(p,q)$ theorem due to Alon and Kleitman \cite{pqthm}. In this case we have $C(3,3) = 1$, which is just Helly's theorem in the plane. Already the case $p=4$ and $q=3$ is much more difficult (the first absolute bound came from the Alon--Kleitman theorem) and the best known bounds are $3\leq C(4,3) \leq 13$ \cite{geza}. 

In the more general setting of Theorem \ref{connected pq}, simple examples show that $C(3,3)>1$, but we are not aware of any examples which show that $C(4,3)>3$. 

Our proof of Theorem \ref{connected pq} is based on a generalization of the $(p,q)$ theorem concerning nerve complexes, due to Alon et al. \cite{transversal}. More precisely, they show that the assertion of the $(p,q)$ theorem holds for all families $\mathcal{F}$ which satisfy a certain ``fractional Helly property''. The connection between this fractional Helly property and nerve complexes is based on combinatorial results due to Kalai \cite{upperbound,shifting} concerning the $f$-vectors of simplicial complexes whose induced subcomplexes have vanishing homology in sufficiently high dimensions. 

This is where our results on connected covers come into play. It is easy to show that any family of open sets as  described in Theorem \ref{connected pq} can be approximated by a connected cover in a planar graph, in the sense that their nerves are isomorphic. Since planar graphs do not have $K_5$ as a minor, Theorem \ref{main-thm} implies that the nerve has vanishing homology in all dimensions greater or equal to 3. Therefore the results of Kalai imply that we have the required fractional Helly property. (Strictly speaking, Kalai's results  only give us Theorem \ref{connected pq} for $p\geq q\geq 4$, but an additional combinatorial argument allows us to reduce the 4 to a 3.)

\medskip

The rest of the paper is organized as follows. In Section \ref{invariant} we establish several basic properties of the homological dimension. In particular, we show that $\gamma(G)$ is minor-monotone (Theorem \ref{l-minor}) and well-behaved under cliques-sums (Theorem \ref{l-sum}). 
Section \ref{planar graphs} contains the proof that $\gamma(G)\leq 2$ for any planar graph $G$, which is the most crucial step in the proof of Theorem \ref{main-thm}.
The proof of Theorem \ref{connected pq} is given in Section \ref{helly-thms}, where we also establish some additional Helly-type theorems for connected covers in graphs, in particular a fractional Helly theorem (Theorem \ref{fractional helly}). Some final remarks are given in Section \ref{final}.

\section{Properties of the homological dimension} \label{invariant}

\subsection{Monotonicity} A simple but important observation is that the homological dimension of a graph is monotone with respect to the minor relation.

\begin{thm} \label{l-minor} 
If $H$ is a minor of $G$, then $\gamma(H)\leq \gamma(G)$.
\end{thm}

\begin{proof} Suppose $H\prec G$. Then there exists a family $\{V_w\}_{w\in V(H)}$ of pairwise disjoint subsets of $V(G)$ such that the induced subgraph $G[V_w]$ is connected for every $w\in V(H)$ and such that for every edge $uw$ in $H$ there is an edge in $G$ connecting a vertex in $V_u$ to a vertex in $V_w$. This induces a map from the subsets of $V(H)$ to the subsets of $V(G)$ which sends a subset $W\subset V(H)$ to the subset $V_W = \bigcup_{w\in W}V_w \subset V(G)$ and which satisfies the property that $G[V_W]$ is connected whenever $H[W]$ is connected. Applying this map to the members of a connected cover $\mathcal{F}$ in $H$ gives us a connected cover $\FF '$ in $G$ such that $N(\FF)$ and $N(\FF ')$ are isomorphic, and the claim follows. \end{proof}

Let $K$ be a simplicial complex and let $x$ be a vertex in $K$. 
We denote by $K-x$ the induced subcomplex of $K$ obtained by deleting the vertex $x$, that is,
\[K-x = \{\sigma \in K :  x\notin \sigma\}.\]
The {\bf{\em star}} of $x$, denoted by $\mbox{st}_K(x)$, is the subcomplex of $K$ defined as 
\[\mbox{st}_K(x)  =  \{\sigma \in K :  \sigma \cup \{x\} \in K\}.\]
Finally, the {\bf{\em link}} of $x$, denoted by $\mbox{lk}_K(x)$, is the subcomplex of $K$ defined as 
\[\mbox{lk}_K(x) =  \{\sigma \in K  :  x\notin \sigma,  \sigma\cup\{x\} \in K\} = \mbox{st}_K(x) - x.\] 
Note that we have the identities \[K = \mbox{st}_K(x) \cup (K-x)\; \text{ and } \; \mbox{lk}_K(x) = \mbox{st}_K(x)\cap (K-x).\] By applying the Mayer--Vietoris exact sequence to the pair $\mbox{st}_K(x)$ and $K-x$, and noting that $\mbox{st}_K(x)$ is contractible, we obtain the exact sequence
\begin{equation}\label{star-link}
\cdots \to \tilde{H}_{j}(\mbox{lk}_K(x)) \to  \tilde{H}_{j}(K-x) \to \tilde{H}_{j}(K) \to \tilde{H}_{j-1}(\mbox{lk}_K(x)) \to \cdots. 
\end{equation}

As an application of \eqref{star-link} we have the following.

\begin{prop} \label{dLeray}
Let $G$ be a graph with  $\gamma(G) \geq 0$. For every $0\leq j\leq \gamma(G)$ there exists a connected cover $\FF$ in $G$ such that  $\tilde{H}_{j}(N(\FF)) \neq 0$.
\end{prop}

\begin{proof} For $\gamma(G) = 0$ there is nothing to show, so let $\FF$ be a connected cover in $G$ and suppose $\tilde{H}_j(N(\FF))\neq 0$ for some $0<j\leq \gamma(G)$. We will construct a connected cover $\FF_x$ in $G$ such that $\tilde{H}_{j-1}(N(\FF_x))\neq 0$.

By deleting members from $\FF$ (if necessary) we may assume that $\FF$ is {\em minimal} in the sense that  $\tilde{H}_j(N(\FF')) = 0$ for every proper subfamily $\FF' \subset \FF$. Let $K = N(\FF)$ and fix a vertex  $x$ in $K$ corresponding to the member $G_x\in \FF$. Now define the family
\[\FF_x  = \{G_i\cap G_x: G_i\in\mathcal{F} \setminus\{G_x\}, G_i\cap G_x \neq \emptyset\},\]
and note that $\FF_x$ is a connected cover in $G_x$ and therefore also a connected cover in $G$. We observe that
\[K-x = N(\mathcal{F} \setminus \{G_x\})  \; \text{ and } \;  \mbox{lk}_K(x) = N(\FF_x),\]
and that $\tilde{H_j}(K-x) = 0$ by the minimality of $\FF$. Therefore the exact sequence \eqref{star-link} implies 
that there is an injection from $\tilde{H}_j(K)$ to $\tilde{H}_{j-1}(\mbox{lk}_K(x))$, and so $\tilde{H}_{j-1}(N(\FF_x))\neq 0$.
\end{proof}

By applying the exact sequence \eqref{star-link} as in the proof of Proposition \ref{dLeray}, one can prove the following. (We leave the details to the reader.)

\begin{cor} \label{leray-cor}
Let $\mathcal{F}$ be a connected cover in $G$ and suppose $\tilde{H}_j(N(\FF))\neq 0$ for some $j\geq 0$. If $\mathcal{F}$ is minimal in the sense that $\tilde{H}_{j}(N(\FF ')) = 0$ for every proper subfamily $\FF' \subset \FF$, then $\gamma(G_x) \geq j-1$ for every member $G_x\in \FF$.
\end{cor}

In some situations we can apply Corollary \ref{leray-cor} to determine the homological dimension of a given graph. We will give several examples throughout this section, starting here with the complete graphs. 

\begin{example} \label{l-complete}
Let $K_n$ denote the complete graph on $n$ vertices. We then have
\[\gamma(K_n) = n-2.\]
For $n=1$ this is by definition. For $n\geq 2$, the family of all induced subgraphs on $n-1$ vertices is a connected cover of $K_n$ and its nerve is the boundary of the $(n-1)$-dimensional simplex, and so $\gamma(K_n) \geq n-2$. 

The upper bound follows by induction on $n$. Consider a minimal size connected cover $\FF$ in $K_n$ realizing $\gamma(K_n)$, in the sense that $\tilde{H}_d(N(\FF))\neq 0$ for $d = \gamma(K_n)$.
Then every member of $\FF$ must be a proper induced subgraph of $K_n$ (or else $N(\FF)$ is contractible). By induction we have  $\gamma(G_x)\leq n-3$ for every member $G_x\in \FF$, and so $d\leq n-2$ by Corollary \ref{leray-cor}. 
\end{example}

\smallskip

\subsection{Clique-sums}
Let $G$ be a graph on the vertex set $V$ and let $V_1$ and $V_2$ be subsets of $V$ such that $V_1\cup V_2 = V$ and $V_1\cap V_2 \neq\emptyset$. If the induced subgraph $G[V_1\cap V_2]$ is a clique (complete subgraph) and there is no edge in $G$ connecting a vertex in $V_1\setminus V_2$ to a vertex in $V_2\setminus V_1$, then $G$ is called a {\bf{\em clique-sum}} of the graphs $G[V_1]$ and $G[V_2]$.\footnote{Some sources allow for the deletion of edges from the clique in the clique-sum operation, but we do not admit this in our definition. Also we include the condition $V_1\cap V_2\neq\emptyset$ since we are dealing only with connected graphs.}

\begin{thm} \label{l-sum}
If $G$ is a clique-sum of graphs $A$ and $B$, then $\gamma(G) = \max \{\gamma(A), \gamma(B)\}$.
\end{thm}

\begin{proof}
The statement is clearly true in the case when $G$ is a tree. So assume $\gamma(G) = d \geq 1$. We will show that there exists connected covers $\mathcal{F}_A$ in $A$ and $\mathcal{F}_B$ in $B$ such that $\tilde{H}_d(N(\mathcal{F}_A)) \oplus \tilde{H}_d(N(\mathcal{F}_B))\neq 0$. This implies that $\gamma(G) \leq \max \{\gamma(A), \gamma(B)\}$, which completes the proof since $\gamma(G) \geq \max \{\gamma(A), \gamma(B)\}$  by Theorem~\ref{l-minor}.

\medskip


Let $\FF$ be a connected cover in $G$ such that $\tilde{H}_d(N(\FF))\neq 0$.
Recall that each member in $\FF$ is an induced subgraph of $G$.
Consider families $\mathcal{F}_A$ and $\mathcal{F}_B$ of induced subgraphs of $G$ defined as 
\[\mathcal{F}_A = \{G_i \cap A : G_i\in \mathcal{F},
G_i\cap A\neq\emptyset\} \; \text{ and } \; \mathcal{F}_B = \{G_i \cap B : G_i\in \mathcal{F}, G_i\cap B\neq\emptyset\}.\]
We claim that $\mathcal{F}_A$ and $\mathcal{F}_B$ are connected covers in $A$ and $B$, respectively.
To see this, consider a pair of vertices $u$ and $v$ in $G_\sigma \cap A$. Since $\mathcal{F}$ is a connected cover, there exists a $uv$-path $\pi$ in $G_\sigma$. If $\pi$ is contained in $A$ we are done. Otherwise, let $x$ be the first vertex of $\pi$ contained in $B$ (as we traverse $\pi$ from $u$ to $v$) and let $y$ be the last vertex of $\pi$ contained in $B$. Clearly $x$ and $y$ are contained in $A\cap B$ which is a clique. Therefore $x$ and $y$ are adjacent and we can reroute the path $\pi$ to find a $uv$-path in $G_\sigma \cap A$. The same argument shows that $\mathcal{F}_B$ is a connected cover in $B$. 

\medskip

Let $K = N(\mathcal{F})$, $K_A = N(\mathcal{F}_A)$, and $K_B = N(\mathcal{F}_B)$. Since there are obvious inclusions from $\mathcal{F}_A$ and $\mathcal{F}_B$ into $\mathcal{F}$ we may regard $K_A$ and $K_B$ as subcomplexes of $K$.  Clearly we have $K = K_A \cup K_B$, so by applying the Mayer-Vietoris exact sequence with the pair $K_A$ and $K_B$ we obtain
\begin{equation}\label{exact111}
0 \to \tilde{H}_{d}(K_A\cap K_B) \to \tilde{H}_d(K_A) \oplus \tilde{H}_d(K_B) \to \tilde{H}_{d}(K) \to \tilde{H}_{d-1}(K_A\cap K_B)  \to \cdots.
\end{equation}
Note that if $\tilde{H}_d(K_A\cap K_B) \neq 0$, then \eqref{exact111} implies that  
 $\tilde{H}_d(K_A) \oplus \tilde{H}_d(K_B) \neq 0$, and we are done. So we may assume $\tilde{H}_d(K_A\cap K_B) = 0$. Observe also that if
$\tilde{H}_{d-1}(K_A\cap K_B) = 0$, then \eqref{exact111} implies that 
 $\tilde{H}_d(K_A) \oplus \tilde{H}_d(K_B)$ surjects onto $\tilde{H}_d(K)$, and again we are done. So we may assume $\tilde{H}_{d-1}(K_A \cap K_B) \neq 0$.

\medskip

Let $G_x$ be the clique $A\cap B$.
Note that if a vertex $i$ appears both in $K_A$ and $K_B$,
then the corresponding graph $G_i \in \FF$ must contain a vertex of $G_x$.
It therefore follows that $G_x\notin \FF$, or else $K_A\cap K_B$ would be a cone with apex $x$, which is contractible,
 contradicting $\tilde{H}_{d-1}(K_A\cap K_B) \neq 0$.
Now we will find new connected covers $\mathcal{F}'_A$ and $\mathcal{F}'_B$ in $A$ and $B$, respectively, such that $\tilde{H}_d(N(\mathcal{F}'_A)) \oplus \tilde{H}_d(N(\mathcal{F}'_B))\neq 0$.
Consider the connected cover $\mathcal{F}'$ in $G$ defined as 
\[\mathcal{F}' = \mathcal{F} \cup \{G_x\}.\]
Let $K' = N(\mathcal{F}')$ and note that there is an obvious inclusion $K\subset K'$. If we let $x$ denote the vertex in $K'$ which represents $G_x$, we then have 
\[K' = K \cup \mbox{st}_{K'}(x) \; \text{ and } \; \mbox{lk}_{K'}(x) = K \cap \mbox{st}_{K'}(x) = K_A \cap K_B.\] 
Applying the exact sequence \eqref{star-link} with the pair $\mbox{st}_{K'}(x)$ and $K'-x = K$, it follows from the assumption  $\tilde{H}_d(K_A\cap K_B)=0$ 
that there is an injection from $\tilde{H}_d(K)$ to $\tilde{H}_d(K')$. So we have  $\tilde{H}_d(K') \neq 0$.

As above, by restricting $\mathcal{F}'$ to $A$ and $B$, we get connected covers $\mathcal{F}'_A$ and $\mathcal{F}'_B$ in $A$ and $B$, respectively, defined as
\[\mathcal{F}'_A = \mathcal{F}_A \cup \{G_x\} \; \text{ and } \; \mathcal{F}'_B = \mathcal{F}_B \cup \{G_x\}.\]
Let $K'_A = N(\mathcal{F}'_A)$ and $K'_B = N(\mathcal{F}'_B)$ which can be regarded as subcomplexes of $K'$. We then have 
\[K' = K'_A \cup K'_B \; \text{ and } \; \mbox{st}_{K'}(x) = K'_A \cap K'_B,\] 
which implies that $K'_A \cap K'_B$ is contractible.
Applying the exact sequence \eqref{exact111} with the pair $K'_A$ and $K'_B$ then implies that $\tilde{H}_d(K')$ and $\tilde{H}_d(K'_A) \oplus \tilde{H}_d(K'_B)$ are isomorphic.
\end{proof}

We can now prove part \eqref{k4minor} of Theorem \ref{main-thm}.

\begin{cor} \label{part 2 main}
Let $G$ be a graph with $\gamma(G)\geq 2$. Then $K_4\prec G$. 
\end{cor}

\begin{proof}
Suppose $G$ does not have a $K_4$-minor. In this case, it is well-known that $G$ is a subgraph of a graph obtained by repeated clique-sums of $K_3$ (see e.g. \cite[Theorem 7.3.1]{diestel}). By Example \ref{l-complete} we have $\gamma(K_3) = 1$, and so Theorem \ref{l-sum} implies that $\gamma(G)\leq 1$.
\end{proof}

\begin{example}\label{compl-mult}
Let $t_1\geq \cdots \geq t_r$ be positive integers and let $K_{t_1, \dots, t_r}$ denote the complete multipartite graph on vertex classes $V_1, \dots, V_r$ where $|V_i|=t_i$. Denote by $m_{t_1, \dots, t_r}$ the order of the largest complete minor in $K_{t_1, \dots, t_r}$, that is, 
\[m_{t_1, \dots, t_r}  = \max\{ s: K_s \prec K_{t_1, \dots, t_r}\}. \]  
Here we show that for the complete multipartite graph $K_{t_1, \dots, t_r}$ we have
\[\gamma(K_{t_1, \dots, t_r}) = m_{t_1, \dots, t_r} - 2.\]

Since $\gamma(K_n) = n-2$, it follows from  Theorem \ref{l-minor} that $\gamma(K_{t_1, \dots, t_r})\geq m_{t_1, \dots, t_r} - 2$. For the upper bound we proceed by induction on $t = t_2 + \cdots + t_r$. The base case of the induction is when $t_i = 1$ for all $i\geq 2$. In this case $K_{t_1, \dots, t_r}$ is a clique-sum of $K_r$'s, so by Theorem \ref{l-sum} we have $\gamma(K_{t_1, \dots, t_r}) = r-2 = m_{t_1, \dots, t_r}-2$. 

Now suppose $t_2> 1$ and let $\FF$ be a connected cover in $K_{t_1, \dots, t_r}$ of minimal size which realizes $\gamma(K_{t_1, \dots, t_r})$. We claim there is a pair of adjancent vertices $u,v \in K_{t_1, \dots, t_r}$ and a member $G_x\in \FF$ such that $G_x\subset K_{t_1, \dots, t_r} -  \{u,v\}$. To see this, first note that for every vertex $v$ there is a member of $\FF$ which misses $v$ (or else $N(\FF)$ is contractible). Now, if every member in $\FF$ which misses a vertex in $V_2\cup \cdots \cup V_r$ only misses non-adjacent vertices, then the intersection of all such members is the vertex class $V_1$, which is disconnected. This establishes the claim, and it follows that there exists a subgraph $K_{s_1, \dots, s_q}$, obtained from $K_{t_1, \dots, t_r}$ by deleting a pair of adjacent vertices, which contains a member $G_x \in \FF$. So by  Corollary \ref{leray-cor} and Theorem \ref{l-minor} we have \[\gamma(K_{t_1, \dots, t_r}) \leq \gamma(G_x) + 1 \leq \gamma(K_{s_1, \dots, s_q})+1.\]

The final observation (which we leave to the reader) is that
if $K_{s_1, \dots, s_q}$ is obtained from $K_{t_1, \dots, t_r}$ by deleting a pair of adjacent vertices, then \[s_2+\cdots + s_q < t_2+\cdots + t_r \;\; \mbox{ and } \;\;  m_{s_1, \dots, s_q} < m_{t_1, \dots, t_r}.\]
Therefore, by induction we have $\gamma(K_{t_1, \dots, t_r}) \leq m_{s_1, \dots, s_q} - 1 \leq m_{t_1, \dots, t_r}-2$. 
\end{example}

\begin{example} \label{w8}
The Wagner graph $W_8$ is the non-planar graph obtained from the cycle on 8 vertices by joining each antipodal pair of vertices by an edge. (See Figure \ref{fig:wagner}.)  Here we show that 
\[\gamma(W_8) = 2.\]

\begin{figure}[ht]
\centering
\begin{tikzpicture}[scale = .8, rotate=1*45/2+90]
\begin{scope}[xshift = 5cm]
\coordinate (V1) at (1*360/8+0:2cm);
\coordinate (V2) at (2*360/8-0:2cm);
\coordinate (V3) at (3*360/8+0:2cm);
\coordinate (V4) at (4*360/8-0:2cm);
\coordinate (V5) at (5*360/8+0:2cm);
\coordinate (V6) at (6*360/8-0:2cm);
\coordinate (V7) at (7*360/8+0:2cm);
\coordinate (V8) at (8*360/8-0:2cm);

\draw (V1)--(V2)--(V3)--(V4)--(V5)--(V6)--(V7)--(V8)--cycle;
\draw 
(V1) .. controls (2*360/8:.3cm) and (4*360/8:.3cm).. (V5) 
(V2) .. controls (3*360/8:.3cm) and (5*360/8:.3cm).. (V6) 
(V3) .. controls (4*360/8:-.3cm) and (6*360/8:-.3cm).. (V7) 
(V4) .. controls (5*360/8:.3cm) and (7*360/8:.3cm).. (V8);
\fill (V1) circle (2.3pt);
\fill (V2) circle (2.3pt);
\fill (V3) circle (2.3pt);
\fill (V4) circle (2.3pt);
\fill (V5) circle (2.3pt);
\fill (V6) circle (2.3pt);
\fill (V7) circle (2.3pt);
\fill (V8) circle (2.3pt);
\end{scope}
\end{tikzpicture}
\caption{The Wagner graph $W_8$. \label{fig:wagner}}
\end{figure}
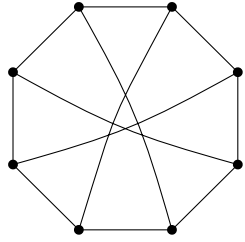

Since $W_8$ contains $K_4$ as a minor, we have  $\gamma(W_8)\geq 2$. To show the reverse inequality, observe that the only proper induced subgraphs of $W_8$ which contain $K_4$ as a minor are those on 7 vertices. If $\gamma(W_8)>2$, then Proposition~\ref{dLeray} implies that there exists a connected cover $\mathcal{F}$ in $W_8$ such that $\tilde{H}_3(N(\FF))\neq 0$. We may assume $\FF$ is minimal with this property, so by Corollary~\ref{leray-cor} we have $\gamma(G_i)\geq 2$ for every $G_i \in \mathcal{F}$. Therefore Corollary \ref{part 2 main} implies that every member of $\FF$ has a $K_4$ minor, which means that every member of $\FF$ is an induced subgraph on 7 vertices. But then the intersection of any 5 members in $\mathcal{F}$ is non-empty, and so the $4$-skeleton of $N(\mathcal{F})$ is complete, implying that $\tilde{H}_3 (N(\mathcal{F})) = 0$. Therefore $\gamma(W_8) =2$.
\end{example}

\begin{remark} \label{equivalent conjecture}
We can give a reformulation of Conjecture \ref{minor-nerve} in terms forbidden minors as follows. Given a positive integer $d$, let $\Gamma_d$ be the set of graphs defined as \[\Gamma_d = \{G: \gamma(G) < d\}.\] Theorem~\ref{l-minor} implies that $\Gamma_d$ is minor-closed, meaning that if $G\in \Gamma_d$ and $H\prec G$, then $H\in \Gamma_d$. The celebrated graph minor theorem of 
Robertson and Seymour 
\cite{graph minor theorem} implies that  $\Gamma_d$ is characterized by a finite list of forbidden minors $F_d$, and Conjecture~\ref{minor-nerve} is equivalent to the statement $F_d = \{K_{d+2}\}$ for all $d\geq 1$.
\end{remark}

\subsection{Proof of Theorem \ref{main-thm}} \label{sec 2.1} Part \eqref{k3minor} follows from Borsuk's nerve theorem and the argument was given in Section \ref{intros} (immediately after the statement of Theorem \ref{main-thm}), while part \eqref{k4minor} was established in Corollary \ref{part 2 main}. For part \eqref{k5minor} we need the following crucial result which we prove in the next section.

\begin{thm} \label{planar}
For any planar graph $G$ we have $\gamma(G)\leq 2$.
\end{thm}

To complete the proof of Theorem \ref{main-thm}, suppose $G$ does not contain $K_5$ as a minor. By a well-known theorem of Wagner \cite{wagner} (see also \cite[Theorem 7.3.4]{diestel}), then $G$ is a subgraph of a graph obtained by repeated clique-sums of planar graphs and the Wagner graph $W_8$. Since the homological dimensions of these graphs are at most 2 (Example \ref{w8} and Theorem \ref{planar}), Theorem \ref{l-sum} implies that $\gamma(G)\leq 2$.

\section{Connected covers in planar graphs} \label{planar graphs}
Here we give a proof of Theorem \ref{planar}. By Proposition \ref{dLeray} it suffices focus on the homology in dimension 3, and so our goal is 
to show the following.

\begin{thm} \label{3homol}
Let $\mathcal{F}$ be a connected cover in a planar graph $G$. Then $\tilde{H}_3(N(\mathcal{F})) = 0$.
\end{thm}

Let us briefly describe the strategy of the proof. Without loss of generality we may assume $G = \bigcup_{G_i\in \FF} G_i$. The basic idea is to add vertices from $G$ to some fixed member $G_x\in \mathcal{F}$ in a systematic manner, while maintaining the property of being a connected cover. This will affect the nerve of $\mathcal{F}$, but we will show that if the 3-dimensional homology of $N(\mathcal{F})$ is non-zero, then it does not vanish during this process. On the other hand, the process can be carried out until $G_x = G$, in which case the nerve is contractible (it is a cone with apex $x$) and its reduced homology is zero in all dimensions. 

\subsection{Connectedness} \label{sec:connect}  It is convenient to assume that $G$ and the members of $\mathcal{F}$ are sufficiently connected. We say that a connected cover $\mathcal F$ in $G$ is {\bf {\em 2-connected}} if $G$ and every member of $\mathcal{F}$ are 2-connected. (This guarantees that faces in an embedding of $G$ are bounded by cycles in $G$.)

\begin{lem} \label{2connected}
Let $\mathcal{F}$ be a connected cover in a connected planar graph $G$. Then there exists a $2$-connected cover ${\mathcal F}'$ in a planar graph $G'$ such that $G\prec G'$ and $N(\mathcal{F})\cong N(\mathcal{F}')$.
\end{lem}

\begin{proof}
We first construct $G'$. Fix a planar embedding of $G$ where the vertices are points in general position and the edges are straight line segments. (This is possible by F{\'a}ry's theorem \cite{fary}.) For every vertex $v$, draw a circle $S_v$ of radius $\epsilon>0$ centered at $v$. We choose $\epsilon$ sufficiently small such that no three circles can be intersected by a line. In particular any two circles are pairwise disjoint. For each edge $e$ incident to the vertex $v$, choose a closed arc $A_{(v,e)}\subset S_v$ whose interior intersects $e$, such that $A_{(v,e)}\cap A_{(v,e')} = \emptyset$ for any pair of distinct edges $e$ and $e'$.

The vertex set of $G'$ will consist of the endpoints of the arcs $A_{(v,e)}$. In this way, a vertex in $G$ of degree $k$ gives rise to a set of $2k$ vertices in $G'$. The edge set of $G'$ is constructed as follows. Put edges between the vertices on the circle $S_v$ such that they form a cycle according to their cyclic order on $S_v$ (or just a single edge if the circle contains only a single arc corresponding to a vertex of degree one in $G$). Finally, for an edge $e=vw$ in $G$ which intersects the arcs $A_{(v,e)}\subset S_v$ and $A_{(w,e)}\subset S_w$, we add two more edges to $G'$ by matching up the endpoints of the arcs by a pair of non-crossing segments. (See Figure \ref{fig:thickening}.)

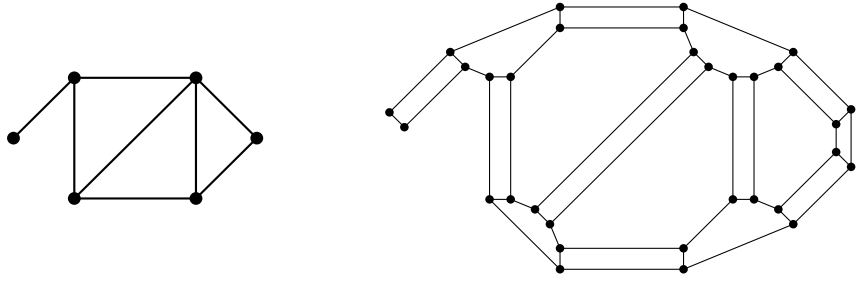
\begin{figure}[ht]
\centering
\begin{tikzpicture}[scale = .8]
\begin{scope}[xshift = 5cm]
\coordinate (A) at (0,0);
\coordinate (A1) at ($(A)+(-10:1cm)$);
\coordinate (A2) at ($(A)+(10:1cm)$);
\coordinate (A3) at ($(A)+(35:1cm)$);
\coordinate (A4) at ($(A)+(55:1cm)$);
\coordinate (A11) at ($(A)+(80:1cm)$);
\coordinate (A12) at ($(A)+(100:1cm)$);

\coordinate (B) at (4,0);
\coordinate (B1) at ($(B)+(190:1cm)$);
\coordinate (B2) at ($(B)+(170:1cm)$);
\coordinate (B3) at ($(B)+(35:1cm)$);
\coordinate (B4) at ($(B)+(55:1cm)$);
\coordinate (B7) at ($(B)+(80:1cm)$);
\coordinate (B8) at ($(B)+(100:1cm)$);

\coordinate (C) at (6,2);

\coordinate (D) at (4,4);
\coordinate (D5) at ($(D)+(-35:1cm)$);
\coordinate (D6) at ($(D)+(-55:1cm)$);
\coordinate (D7) at ($(D)+(-80:1cm)$);
\coordinate (D8) at ($(D)+(-100:1cm)$);
\coordinate (D3) at ($(D)+(235:1cm)$);
\coordinate (D4) at ($(D)+(215:1cm)$);
\coordinate (D9) at ($(D)+(170:1cm)$);
\coordinate (D10) at ($(D)+(190:1cm)$);

\coordinate (C5) at ($(D5)+(.95,-.95)$);
\coordinate (C6) at ($(D6)+(.95,-.95)$);
\coordinate (C3) at ($(B3)+(.95,.95)$);
\coordinate (C4) at ($(B4)+(.95,.95)$);

\coordinate (E) at (0,4);
\coordinate (E9) at ($(E)+(10:1cm)$);
\coordinate (E10) at ($(E)+(-10:1cm)$);
\coordinate (E11) at ($(E)+(-80:1cm)$);
\coordinate (E12) at ($(E)+(-100:1cm)$);
\coordinate (E1) at ($(E)+(-125:1cm)$);
\coordinate (E2) at ($(E)+(-145:1cm)$);

\coordinate (F) at (-2,2);
\coordinate (F1) at ($(E1)+(-1,-1)$);
\coordinate (F2) at ($(E2)+(-1,-1)$);

\draw (A1)--(A2)--(A3)--(A4)--(A11)--(A12)--cycle;
\draw (B2)--(B1)--(B3)--(B4)--(B7)--(B8)--cycle;
\draw (C5)--(C6)--(C4)--(C3)--cycle;
\draw (D9)--(D10)--(D4)--(D3)--(D8)--(D7)--(D6)--(D5)--cycle;
\draw (E1)--(E2)--(E9)--(E10)--(E11)--(E12)--cycle;
\draw (F1)--(F2);

\draw (A1)--(B1) (A2)--(B2);
\draw (B3)--(C3) (B4)--(C4);
\draw (C5)--(D5) (C6)--(D6);
\draw (B7)--(D7) (B8)--(D8);
\draw (A3)--(D3) (A4)--(D4);
\draw (D9)--(E9) (D10)--(E10);
\draw (A11)--(E11) (A12)--(E12);
\draw (E1)--(F1) (E2)--(F2);

\fill (A1) circle (2pt);
\fill (A2) circle (2pt);
\fill (A3) circle (2pt);
\fill (A4) circle (2pt);
\fill (A11) circle (2pt);
\fill (A12) circle (2pt);

\fill (B1) circle (2pt);
\fill (B2) circle (2pt);
\fill (B3) circle (2pt);
\fill (B4) circle (2pt);
\fill (B7) circle (2pt);
\fill (B8) circle (2pt);

\fill (C3) circle (2pt);
\fill (C4) circle (2pt);
\fill (C5) circle (2pt);
\fill (C6) circle (2pt);

\fill (D3) circle (2pt);
\fill (D4) circle (2pt);
\fill (D5) circle (2pt);
\fill (D6) circle (2pt);
\fill (D7) circle (2pt);
\fill (D8) circle (2pt);
\fill (D9) circle (2pt);
\fill (D10) circle (2pt);

\fill (E1) circle (2pt);
\fill (E2) circle (2pt);
\fill (E9) circle (2pt);
\fill (E10) circle (2pt);
\fill (E11) circle (2pt);
\fill (E12) circle (2pt);
\fill (F1) circle (2pt);
\fill (F2) circle (2pt);
\end{scope}

\begin{scope}[xshift = -2cm, yshift = 1cm]
\draw[thick] (2,0)--(3,1)--(2,2)--(2,0) (2,2)--(0,2) (2,2)--(0,0);
\draw[thick] (-1,1)--(0,2)--(0,0)--(2,0);

\fill (0,0) circle (3pt);
\fill (2,0) circle (3pt);
\fill (3,1) circle (3pt);
\fill (2,2) circle (3pt);
\fill (0,2) circle (3pt);
\fill (-1,1) circle (3pt);
\end{scope}

\end{tikzpicture}
\caption{Thickening a graph to make it 2-connected.  \label{fig:thickening}}
\end{figure}

It is easy to see that $G\prec G'$. There is a natural partition of the vertices of $G'$ according to which circle $S_v$ a vertex in $G'$ belongs to. By construction, the subgraph induced by the vertices on a fixed circle is connected,  and for any edge $e=vw$ in $G$ there is an edge connecting a vertex in $S_v$ to a vertex in $S_w$. This remains true even if we delete an arbitrary vertex from $G'$, or in other words, $G\prec G'-v$ for any vertex $v\in G'$. Since $G$ is connected, it follows that $G'$ is 2-connected. 

Given a connected cover $\mathcal{F}$ in $G$, we can embed $\mathcal{F}$ into $G'$ as in the proof of Theorem \ref{l-minor} to obtain a connected cover $\mathcal{F}'$ of $G'$. Thus, for a member of $G_x = G[V_x] \in \mathcal{F}$, the corresponding member of $\mathcal{F}'$ is the subgraph of $G'$ induced by the set of vertices belonging to the circles $\{S_v  :  v\in V_x\}$. The same argument as before shows that each member of $\mathcal{F}'$ is 2-connected.
\end{proof}

\subsection{Filling faces}\label{sec:filling} 
Let $G$ be a 2-connected planar graph with a fixed embedding in $\mathbb{S}^2$ and let $H = G[W]$ be a 2-connected induced subgraph of $G$. (Note that this fixes an embedding of $H$ as well.) Consider a face $f$ of $H$ and let $U$ denote the subset of vertices of $G$ which are contained in the interior of $f$. (If $H$ is a cycle, then just choose one of the two faces bounded by $f$). Observe that $f$ is also a face of $G$ if and only if $U = \emptyset$.

Define the graph $H'$ to be the induced subgraph obtained by adding the vertices of $U$ to $H$, that is,  
\[H' = G[W\cup U].\]
We say that $H'$ is obtained from $H$ by {\bf{\em filling the face$\bm{f}$}}. 

\begin{lem}\label{fill-conn}
Let $G$ be a 2-connected planar graph embedded in $\mathbb{S}^2$ and let $H$ be a 2-connected induced subgraph of $G$. If $H'$ is obtained from $H$ by filling a face, then $H'$ is 2-connected.
\end{lem}

\begin{proof}
We need to show that for any distinct vertices $u,v,w$ in $H'$ there is a $uv$-path in $H'$ which does not pass through $w$. If $u$ and $v$ both belong to $H$, then this follows from the 2-connectedness of $H$. Now consider the case when $u$ is a vertex in the interior of the face $f$ of $H$ which is being filled. By the 2-connectedness of $G$, there exists a $uv$-path $\pi$ (in $G$) which avoids $w$. If $\pi$ is contained in the interior of $f$, then we are done. Otherwise, let $u'$ be the first vertex of $\pi$ which is not in the interior of the face $f$. It follows that $u'$ is a vertex in the cycle of $G$ bounding the face $f$ and therefore $u'$ belongs to $H$. Let $\pi_1$ be the initial part of $\pi$ connecting $u$ to $u'$. Note that $\pi_1$ is in $H'$. If $v$ is a vertex of $H$, then we are done by concatenating $\pi_1$ with a $u'v$-path in $H$ which avoids $w$. If $v$ is in the interior of $f$, by the same argument as before, there is path in $H'$ avoiding $w$, connecting $v$ to a vertex $v'$ in the cycle bounding $f$. Now, $u'$ and $v'$ can be connected by a path in $H$ avoiding $w$.
\end{proof}

Consider a 2-connected cover $\mathcal{F}$ in $G$. Fix a member $G_x\in \mathcal{F}$ and let $G_x'$ be the graph obtained from $G_x$ by filling a face. Replacing $G_x$ by $G_x'$, we obtain a new family
 \[\mathcal{F}' = (\mathcal{F}\setminus \{G_x\}) \cup \{G'_x\}.\]
We say $\mathcal{F}'$ is obtained from $\mathcal{F}$ by filling a face. Note that there is an obvious inclusion $N(\mathcal{F})\subset N(\mathcal{F}')$.

\begin{lem}\label{filling is connected}
Let $\mathcal{F}$ be a $2$-connected cover in a planar graph $G$ embedded in $\mathbb{S}^2$. If $\mathcal{F}'$ is obtained from $\mathcal{F}$ by filling a face, then $\mathcal{F}'$ is a $2$-connected cover of $G$. 
\end{lem}

\begin{proof}
Let $f$ be the face of $G_x = G[V_x]\in \mathcal{F}$ that is being filled and let $U$ be the set of vertices in $G$ contained in the interior of $f$. We may assume $U\neq \emptyset$, or else $\mathcal{F} = \mathcal{F}'$ and we are done. As above, we define $G_x' = G[V_x\cup U]$ and $\mathcal{F}' = (\mathcal{F}\setminus \{G_x\})\cup \{G_x'\}$. By Lemma \ref{fill-conn} we know that $G_x'$ is 2-connected, so it suffices to show that $\mathcal{F}'$ is a connected cover of $G$.

Consider a simplex $\sigma \in N(\FF)-x$ and suppose $G_\sigma \cap G'_x \neq \emptyset$. We need to show that $G_\sigma \cap G'_x$ is connected. If $G_\sigma \cap U =\emptyset$, then $G_\sigma \cap G_x' = G_\sigma \cap G_x$ which is
connected since $\FF$ is a connected cover. If $G_\sigma \cap G_x = \emptyset$, then the vertices of $G_\sigma$ are contained in $U$ and $G_\sigma \cap G_x' = G_\sigma$ which is connected. It remains to consider the case when $G_\sigma \cap U\neq \emptyset$ and $G_\sigma\cap G_x\neq \emptyset$. Since $G_\sigma \cap G_x$ is connected it suffices to show that each connected component of $G_\sigma \cap G[U]$ is connected to $G_\sigma \cap G_x$ by some edge in $G$. But this is obvious, because $G_\sigma$ is connected and any path from a vertex in $U$ to a vertex which is not in $U$ must pass through a vertex in the cycle which bounds the face $f$, which is contained in $G_x$. 
\end{proof}

\subsection{Critical cycles}
Let $K$ be a simplicial complex and let $C_d(K)$ denote group of $d$-chains of $K$. Denote by $Z_d(K)$ the group of $d$-cycles, and by $B_d(K)$ the group of $d$-boundaries.

Given a vertex $x\in K$ and a subcomplex $L\subset K-x$, let $K_{(x,L)}$ be the simplicial complex defined as \[K_{(x,L)} = K \cup \{\sigma\cup \{x\} : \sigma\in L\}.\] (In other words,  $K_{(x,L)}$ is the union of $K$ and the cone over $L$ with apex $x$.)

By inclusion, every $d$-chain $\alpha = \sum \sigma_i \in C_d(L)$ can be regarded as a $d$-chain in $C_d(K)$ and in $C_d(K_{(x,L)})$. Moreover, we may form a $(d+1)$-chain $[x,\alpha] \in C_{d+1}(K_{(x,L)})$ defined as 
\[[x,\alpha] = 
\begin{cases}
{\textstyle \sum} \left(\sigma_j \cup \{x\} \right), &\text{for } \alpha \neq 0 ,\\
0, &\text{for } \alpha = 0.
\end{cases}\]
A $d$-cycle $\gamma\in Z_d(K)$ is called {\bf{\em $\bm{(x,L)}$-critical}} if $\gamma \not\in B_d(K)$ and $\gamma \in B_d(K_{(x,L)})$. In other words, $\gamma$ is $(x,L)$-critical if $\gamma$ does not vanish in $\tilde{H}_d(K)$ while  $\gamma$ vanishes in $\tilde{H}_d(K_{(x,L)})$.

\begin{lem} \label{critical}
Every $(x,L)$-critical $d$-cycle in $Z_d(K)$ is homologous to a $d$-cycle of the form
\[\gamma = \partial[x,\beta],\]
where
\begin{enumerate} 
\item $\beta = \sum\sigma_i \in C_d(L)$,
\item $\sigma_i \cup \{x\} \in K_{(x,L)}\setminus K$ for all $i$, and
\item $[x,\partial \beta] \in C_d(K)$.
\end{enumerate}
\end{lem}

\begin{proof}
Suppose $\gamma_0\in Z_d(K)$ is $(x,L)$-critical. Since $\gamma_0 \in B_d(K_{(x,L)})$, we may write $\gamma_0 = \partial \beta_0$, where $\beta_0 \in C_{d+1}(K_{(x,L)})$. Partition the simplices in $\beta_0$ as 
\[\beta_0 = \beta_{\text{in } K} + \beta_{\text{not in } K},\] 
where the simplices of $\beta_{\text{in } K}$ belong to $K$ and the simplices of $\beta_{\text{not in } K}$ belong to $K_{(x,L)}\setminus K$. Then 
\[\gamma = \gamma_0 + \partial\beta_{\text{in }K} = \partial \beta_{\text{not in }K}\] is an $(x,L)$-critical cycle homologous to $\gamma_0$. By definition, any $(d+1)$-simplex which is in $K_{(x,L)}\setminus K$ is of the form $\sigma \cup\{x\}$ where $\sigma$ is a $d$-simplex in $L$. We may therefore write
\[\beta_{\text{not in } K} = [x,\beta],\] where $\beta = \sum\sigma_i \in C_d(L)$ and $\sigma_i\cup \{x\} \in K_{(x,L)}\setminus K$ for all $i$. Finally, observe that the boundary operator satisfies the ``product rule''
\[\partial [x,\beta] = \beta + [x,\partial\beta].\]
Since $\gamma = \partial[x,\beta] \in Z_d(K) \subset C_d(K)$ and $\beta \in C_d(L)\subset C_d(K)$, it follows that $[x,\partial\beta] \in C_d(K)$.
\end{proof}

\subsection{A combinatorial identity}\label{boundary identity} 

Let $\tau = \sum \tau_i$ be a simplicial 2-cycle and suppose we are given a linear ordering of the simplices in $\tau$, say, 
\[\tau_1 \prec \tau_2 \prec \cdots \prec \tau_n.\]
For any ordered sequence $\tau_i\prec\tau_j\prec\tau_k\prec\tau_l$, we define
\[\eta_{(\tau_i\prec\tau_j\prec\tau_k\prec\tau_l)}= 
\begin{cases}
(\tau_{i} \cap \tau_{k}) \cup (\tau_{j} \cap \tau_{l}), &  \text{when } |(\tau_{i} \cap \tau_{k}) \cup (\tau_{j} \cap \tau_{l})| = 4, \\
0, & \text{otherwise.}
\end{cases}
\]
Note that if $\eta_{(\tau_i\prec\tau_j\prec\tau_k\prec\tau_l)}\neq 0$, then $|\tau_i\cap \tau_k| = |\tau_j\cap \tau_l| = 2$ and $(\tau_i\cap \tau_k) \cap (\tau_j\cap \tau_l) = \emptyset$. We may therefore regard $\eta_{(\tau_i\prec\tau_j\prec\tau_k\prec\tau_l)}$ as a 3-simplex obtained as the join of the disjoint pair of 1-simplices $\tau_i\cap \tau_k$ and $\tau_j\cap \tau_l$. By doing so, we can form a simplicial 3-chain $T_{(\tau, \prec)}$, defined as
\[T_{(\tau, \prec)} = {\textstyle{\sum}} \eta_{(\tau_i\prec\tau_j\prec\tau_k\prec\tau_l)},\] 
where the sum is over all ordered sequences $\tau_i\prec\tau_j\prec\tau_k\prec\tau_l$ of simplices in $\tau$.

\begin{remark} \label{geom int}
Observe that the 3-chain $T_{(\tau,\prec)}$ only depends on the separation relation induced by $\prec$. In other words, any two linear orderings that induce the same (or reverse) cyclic orders will produce the same 3-chain. This observation gives us a natural geometric interpretation which will be convenient later on.  Represent the 2-simplices in $\tau$ by a set of distinct points in on a circle consistent with the cyclic order induced by $\prec$. Connect two points by a segment if the corresponding 2-simplices have a 1-simplex in common and label this segment by the common 1-simplex. If the labels of a pair of crossing segments correspond to a pair of disjoint 1-simplices, then the join of these 1-simplices is a 3-simplex which contributes to $T_{(\tau, \prec)}$. (See Figure \ref{fig:3T}.)

\begin{figure}[ht]
\centering
\begin{tikzpicture}
\begin{scope}[yshift=3.9cm]
\coordinate (A) at (-30:2cm);
\coordinate (B) at (90:2cm);
\coordinate (C) at (210:2cm);
\coordinate (D) at (0.6,-.1);
\coordinate (E) at (-0.2,0.2);
\fill (A) circle (2pt);
\fill (B) circle (2pt);
\fill (C) circle (2pt);
\fill (D) circle (2pt);
\fill (E) circle (2pt);
\draw (A)--(B)--(C)--(A)--(D)--(E)--(B)--(D)--(C)--(E);
\node[right] at (A) {\footnotesize $1$};
\node[above] at (B) {\footnotesize $2$};
\node[left] at (C) {\footnotesize $3$};
\node[right] at (D) {\footnotesize $4$};
\node[left] at (E) {\footnotesize $5$};
\end{scope}

\begin{scope}[xshift = -4cm]
\def \r {1.65};
\coordinate (S1) at (10:\r cm);
\coordinate (S2) at (70:\r cm);
\coordinate (S3) at (130:\r cm);
\coordinate (S4) at (190:\r cm);
\coordinate (S5) at (250:\r cm);
\coordinate (S6) at (310:\r cm);
\coordinate (T1) at ($1.42*(S1)$);
\coordinate (T2) at ($1.2*(S2)$);
\coordinate (T3) at ($1.23*(S3)$);
\coordinate (T4) at ($1.42*(S4)$);
\coordinate (T5) at ($1.2*(S5)$);
\coordinate (T6) at ($1.24*(S6)$);

\node at (T1) {\footnotesize $\{1,2,3\}$};
\node at (T2) {\footnotesize $\{2,4,5\}$};
\node at (T3) {\footnotesize $\{3,4,5\}$};
\node at (T4) {\footnotesize $\{1,3,4\}$};
\node at (T5) {\footnotesize $\{2,3,5\}$};
\node at (T6) {\footnotesize $\{1,2,4\}$};

\fill (S1) circle (2pt);
\fill (S2) circle (2pt);
\fill (S3) circle (2pt);
\fill (S4) circle (2pt);
\fill (S5) circle (2pt);
\fill (S6) circle (2pt);

\draw[name path=S1--S4] (S1)--(S4); 
\draw[name path=S1--S5] (S1)--(S5); 
\draw (S1)--(S6); 
\draw (S2)--(S3);
\draw[name path=S2--S5] (S2)--(S5); 
\draw[name path=S2--S6] (S2)--(S6);
\draw (S3)--(S4); 
\draw[name path=S3--S5] (S3)--(S5);
\draw[name path=S4--S6] (S4)--(S6);

\coordinate (I1) at (intersection of S1--S4 and S2--S5);
\coordinate (I2) at (intersection of S1--S4 and S2--S6);
\coordinate (I3) at (intersection of S1--S5 and S4--S6);
\coordinate (I4) at (intersection of S3--S5 and S4--S6);
\coordinate (I5) at (intersection of S2--S5 and S4--S6);

\fill[magenta] (I1) circle (2.5pt); 
\fill[gray] (I2) circle (2.5pt); 
\fill[gray] (I3) circle (2.5pt); 
\fill[orange] (I4) circle (2.5pt); 
\fill[cyan] (I5) circle (2.5pt); 

\node at (0,-2.5) {$\prec_1$};
\end{scope}
\begin{scope}[xshift = -4cm]
\def \r {1.65};
\coordinate (P2) at ($(S1)+(8,0)$);
\coordinate (P1) at ($(S2)+(8,0)$);
\coordinate (P3) at ($(S3)+(8,0)$);
\coordinate (P4) at ($(S4)+(8,0)$);
\coordinate (P5) at ($(S5)+(8,0)$);
\coordinate (P6) at ($(S6)+(8,0)$);
\coordinate (Q2) at ($(T1)+(8,0)$);
\coordinate (Q1) at ($(T2)+(8,0)$);
\coordinate (Q3) at ($(T3)+(8,0)$);
\coordinate (Q4) at ($(T4)+(8,0)$);
\coordinate (Q5) at ($(T5)+(8,0)$);
\coordinate (Q6) at ($(T6)+(8,0)$);

\node at (Q1) {\footnotesize $\{1,2,3\}$};
\node at (Q2) {\footnotesize $\{2,4,5\}$};
\node at (Q3) {\footnotesize $\{3,4,5\}$};
\node at (Q4) {\footnotesize $\{1,3,4\}$};
\node at (Q5) {\footnotesize $\{2,3,5\}$};
\node at (Q6) {\footnotesize $\{1,2,4\}$};

\fill (P1) circle (2pt);
\fill (P2) circle (2pt);
\fill (P3) circle (2pt);
\fill (P4) circle (2pt);
\fill (P5) circle (2pt);
\fill (P6) circle (2pt);

\draw[name path=P1--P4] (P1)--(P4); 
\draw[name path=P1--P5] (P1)--(P5); 
\draw[name path=P1--P6] (P1)--(P6); 
\draw[name path=P2--P3] (P2)--(P3);
\draw[name path=P2--P5] (P2)--(P5); 
\draw[name path=P2--P6] (P2)--(P6);
\draw[name path=P3--P4] (P3)--(P4); 
\draw[name path=P3--P5] (P3)--(P5);
\draw[name path=P4--P6] (P4)--(P6);

\coordinate (J0) at (intersection of P1--P4 and P2--P3);
\coordinate (J1) at (intersection of P1--P5 and P2--P3);
\coordinate (J2) at (intersection of P1--P6 and P2--P3);
\coordinate (J3) at (intersection of P1--P5 and P4--P6);
\coordinate (J4) at (intersection of P3--P5 and P4--P6);
\coordinate (J5) at (intersection of P2--P5 and P4--P6);

 \fill[orange] (J0) circle (2.5pt); 
 \fill[green!80!black] (J1) circle (2.5pt); 
 \fill[cyan] (J2) circle (2.5pt); 
\fill[gray] (J3) circle (2.5pt); 
\fill[orange] (J4) circle (2.5pt); 
\fill[cyan] (J5) circle (2.5pt); 

\node at (8,-2.5) {$\prec_2$};
\end{scope}

\begin{scope}[yshift = -4.1cm, xshift = -2cm, scale = .63]
\coordinate (V1) at (-.8,0);
\coordinate (V2) at (0.1,-.4);
\coordinate (V3) at (.6,.3);
\coordinate (V4) at (-.2,1);
\fill[orange, opacity = .2] (V1)--(V2)--(V4);
\fill[orange!80!black, opacity = .4] (V2)--(V3)--(V4);
\draw[line width = 0.1mm, orange] (V1)--(V2)--(V3)--(V4)--cycle; 
\draw[orange!70!black] (V1)--(V3);
\draw[thick, orange!70!black] (V2)--(V4);
\node[left] at (V1) {\footnotesize $1$};
\node[below] at (V2) {\footnotesize $3$};
\node[right] at (V3) {\footnotesize $4$};
\node[above] at (V4) {\footnotesize $5$};

\coordinate (A1) at ($(V1) + (4,0)$);
\coordinate (A2) at ($(V2) + (4,0)$);
\coordinate (A3) at ($(V3) + (4,0)$);
\coordinate (A4) at ($(V4) + (4,0)$);
\fill[cyan, opacity = .2] (A1)--(A2)--(A4);
\fill[cyan!80!black, opacity = .4] (A2)--(A3)--(A4);
\draw[line width = 0.1mm, cyan] (A1)--(A2)--(A3)--(A4)--cycle; 
\draw[cyan!70!black] (A1)--(A3);
\draw[thick, cyan!70!black] (A2)--(A4);
\node[left] at (A1) {\footnotesize $1$};
\node[below] at (A2) {\footnotesize $2$};
\node[right] at (A3) {\footnotesize $4$};
\node[above] at (A4) {\footnotesize $5$};

\coordinate (B1) at ($(A1) + (4,0)$);
\coordinate (B2) at ($(A2) + (4,0)$);
\coordinate (B3) at ($(A3) + (4,0)$);
\coordinate (B4) at ($(A4) + (4,0)$);
\fill[magenta, opacity = .2] (B1)--(B2)--(B4);
\fill[magenta!80!black, opacity = .4] (B2)--(B3)--(B4);
\draw[line width = 0.1mm, magenta] (B1)--(B2)--(B3)--(B4)--cycle; 
\draw[ magenta!70!black] (B1)--(B3);
\draw[thick, magenta!70!black] (B2)--(B4);
\node[left] at (B1) {\footnotesize $1$};
\node[below] at (B2) {\footnotesize $2$};
\node[right] at (B3) {\footnotesize $3$};
\node[above] at (B4) {\footnotesize $5$};

\node at (-3.1,.2) {$T_{(\tau, \prec_1)}$};
\node at (-1.8,.2) {$=$};
\node at (1.9,.2) {$+$};
\node at (5.9,.2) {$+$};
\end{scope}

\begin{scope}[yshift = -6.3cm, xshift = -1cm, scale=.63]
\coordinate (Y1) at (-.8,0);
\coordinate (Y2) at (0.1,-.4);
\coordinate (Y3) at (.6,.3);
\coordinate (Y4) at (-.2,1);
\fill[green, opacity = .2] (Y1)--(Y2)--(Y4);
\fill[green!80!black, opacity = .4] (Y2)--(Y3)--(Y4);
\draw[line width = 0.1mm, green] (Y1)--(Y2)--(Y3)--(Y4)--cycle; 
\draw[green!70!black] (Y1)--(Y3);
\draw[thick, green!70!black] (Y2)--(Y4);
\node[left] at (Y1) {\footnotesize $4$};
\node[below] at (Y2) {\footnotesize $3$};
\node[right] at (Y3) {\footnotesize $5$};
\node[above] at (Y4) {\footnotesize $2$};

\coordinate (Z1) at ($(Y1) + (4,0)$);
\coordinate (Z2) at ($(Y2) + (4,0)$);
\coordinate (Z3) at ($(Y3) + (4,0)$);
\coordinate (Z4) at ($(Y4) + (4,0)$);
\fill[gray, opacity = .2] (Z1)--(Z2)--(Z4);
\fill[gray!80!black, opacity = .4] (Z2)--(Z3)--(Z4);
\draw[line width = 0.1mm, gray] (Z1)--(Z2)--(Z3)--(Z4)--cycle; 
\draw[gray!70!black] (Z1)--(Z3);
\draw[thick, gray!70!black] (Z2)--(Z4);
\node[left] at (Z1) {\footnotesize $1$};
\node[below] at (Z2) {\footnotesize $2$};
\node[right] at (Z3) {\footnotesize $4$};
\node[above] at (Z4) {\footnotesize $3$};

\node at (-3.1,.2) {$T_{(\tau, \prec_2)}$};
\node at (-1.8,.2) {$=$};
\node at (1.9,.2) {$+$};
\end{scope}
\end{tikzpicture}
\caption{Two distinct orderings, $\prec_1$ and $\prec_2$, of the 2-simplices of a triangulation of the 2-sphere, and the resulting 3-chains, $T_{(\tau,\prec_1)}$ and $T_{(\tau,\prec_2)}$. \label{fig:3T}}
\end{figure}
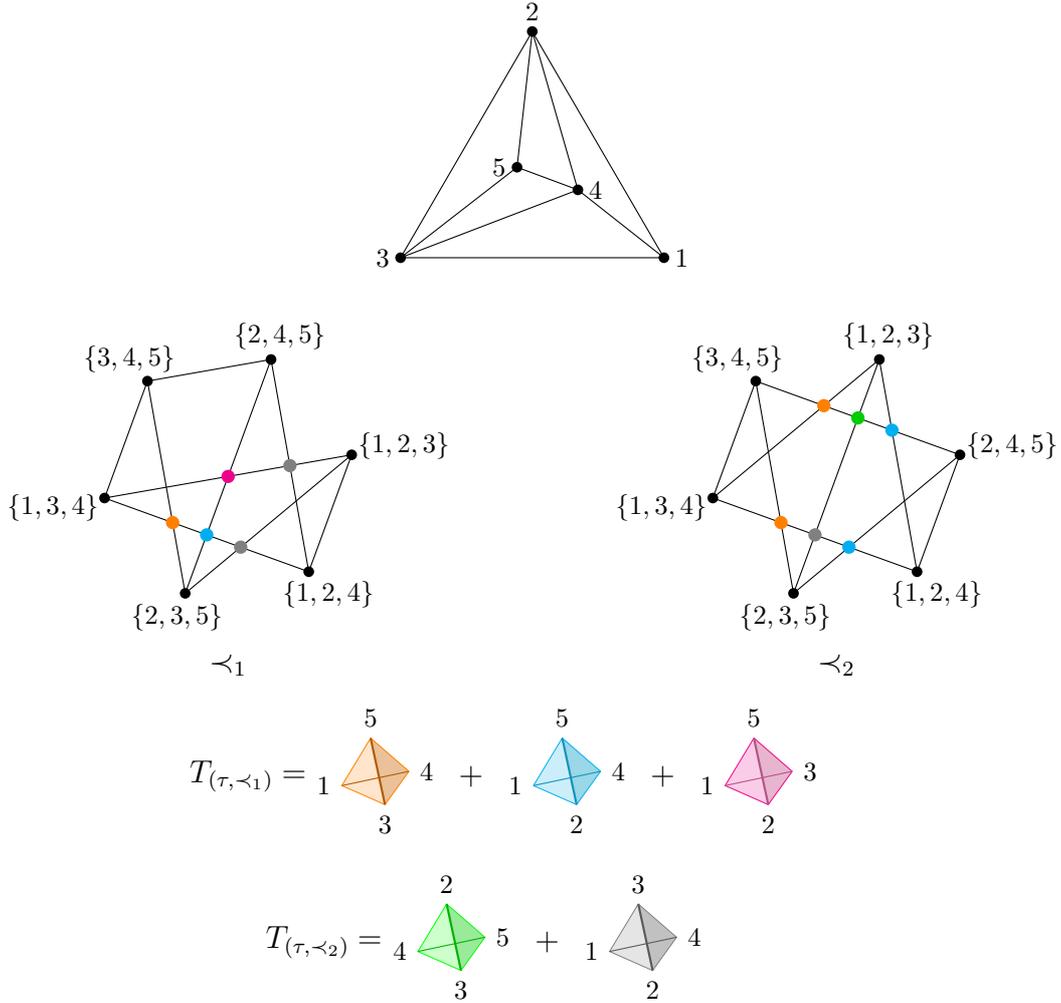
\end{remark}

\begin{lem} \label{isoperm}
Let $\tau$ be a simplicial 2-cycle and $\prec$ a linear ordering of the simplices in $\tau$. Then $\partial T_{(\tau, \prec)} = \tau$.
\end{lem}

\begin{proof} We first show that $\partial T_{(\tau, \prec)}$ is $\prec$-invariant. In other words, for any two linear orderings $\prec_1$ and $\prec_2$ we want to show that $\partial(T_{(\tau,\prec_1)} + T_{(\tau,\prec_2)}) = 0$. It is sufficient to show this in the case when $\prec_1$ and $\prec_2$ differ by a simple transposition of a pair consecutive simplices $\tau_1$ and $\tau_{2}$. Observe that if $\eta_*$ is a 3-simplex in $T_{(\tau,\prec_1)} + T_{(\tau,\prec_2)}$, then $\eta_*$ must be the join of a 1-simplex from $\tau_1$ and a 1-simplex from $\tau_2$, because $\tau_1$ and $\tau_2$ are consecutive in $\prec_1$ and $\prec_2$ (but in opposite order). We claim that the converse also holds; for any $\tau_i$ and $\tau_j$ such that $\tau_1\cap \tau_i$ and $\tau_2\cap \tau_j$ are a pair of disjoint 1-simplices, their join $\eta_* = (\tau_1\cap\tau_i) \cup (\tau_2\cap \tau_j)$ is in either $T_{(\tau,\prec_1)}$ or $T_{(\tau,\prec_2)}$, but not in both. This is easily seen by considering the geometric interpretation given in Remark \ref{geom int}; if two segments cross, then they do not cross after transposing a pair of consecutive endpoints (and vice versa). We now distinguish cases according to the cardinality of $\tau_1\cup \tau_2$.

\medskip

\noindent {\bf Case 1:} $|\tau_1\cup \tau_2| = 4$. For convenience we assume $\tau_1 = \{1,2,3\}$ and $\tau_2 = \{2,3,4\}$.  In this case, the only 3-simplex $\eta_*$ that can arise as the join of a 1-simplex from $\tau_1$ and a 1-simplex from $\tau_2$, is the 3-simplex $\{1,2,3,4\}$. It can arise in two possible ways:
\[\eta_* = \{1,2\} \cup \{3,4\} \;\; \text{ or } \;\; \eta_* = \{1,3\}\cup\{2,4\}.\]
Since $\tau$ is a 2-cycle, we have an odd number of 2-simplices $\tau_i = \{1,2,x\}$ ($x\neq 3$) and an odd number of 2-simplices $\tau_j = \{3,4,y\}$ ($y\neq 2$). Since every possible combination contributes to $T_{(\tau,\prec_1)}+T_{(\tau,\prec_2)}$, the number of contributions of the form $\{1,2\}\cup\{3,4\}$ is odd. The same argument shows that the number of contributions of the form $\{1,3\}\cup \{2,4\}$ is also odd. In total we get an even number of contributions of the 3-simplex $\{1,2,3,4\}$. This means $T_{(\tau,\prec_1)}+T_{(\tau,\prec_2)} = 0$ and in particular $\partial(T_{(\tau,\prec_1)}+T_{(\tau,\prec_2)}) = 0$.

\medskip

\noindent {\bf Case 2:} $|\tau_1\cup \tau_2| = 5$. We may assume that $\tau_1 = \{1,2,3\}$ and $\tau_2 = \{3,4,5\}$. In this case, there are five distinct pairs of disjoint 1-simplices whose joins contribute to $T_{(\tau,\prec_1)}+T_{(\tau,\prec_2)}$. These are
\[\{1,2\}\cup \{3,4\} \; , \; \{1,2\}\cup \{3,5\} \; , \; \{1,2\}\cup\{4,5\}\; , \;\{1,3\}\cup \{4,5\}\; , \;\{2,3\}\cup \{4,5\}. \]
Since $\tau$ is a 2-cycle, the same argument as in {\bf Case 1}, shows that the number of contributions to
$T_{(\tau,\prec_1)}+T_{(\tau,\prec_2)}$ is odd, for each of these 3-simplices. Therefore $T_{(\tau,\prec_1)}+T_{(\tau,\prec_2)}$ equals the sum of the 3-simplices listed above, and it is straightforward to check that $\partial(T_{(\tau,\prec_1)}+T_{(\tau,\prec_2)}) = 0$.

\medskip

\noindent {\bf Case 3:} $|\tau_1\cup \tau_2| = 6$. In this case, $\tau_1$ and $\tau_2$ are disjoint, and the join of any 1-simplex from $\tau_1$ and any 1-simplex from $\tau_2$ will contribute an odd number of times to $T_{(\tau,\prec_1)}+T_{(\tau,\prec_2)}$ (again by the assumption that $\tau$ is a 2-cycle). Thus, $T_{(\tau,\prec_1)}+T_{(\tau,\prec_2)}$ is the sum of all 3-simplices in the join of $\partial \tau_1$ and $\partial \tau_2$ (which is homeomorphic to $\mathbb{S}^3$), and again we see that $\partial(T_{(\tau,\prec_1)}+T_{(\tau,\prec_2)}) = 0$.
  
\medskip

It remains to show that $\partial T_{(\tau, \prec)} = \tau$ for every linear order $\prec$ of the $2$-simplices in $\tau$. 
By the $\prec$-invariance, it is sufficient to find for each triple $\{a,b,c\}$ of $0$-simplices an order $\prec'$ such that $\{a, b, c\}$ is contained in an odd number of $3$-simplices in $T_{(\tau, \prec')}$ if and only if $\{a, b, c\}$ is a $2$-simplex in $\tau$.
To do this, we label and order the 0-simplices by the positive integers so that $\{a,b,c\} = \{1,2,3\}$, and let $\prec_\text{lex}$ denote the lexicographic order on the triples of $0$-simplices. 
This induces a linear ordering of the 2-simplices in $\tau$. 
We will show that the triple $\{1,2,3\}$ is contained in an odd number of 3-simplices in $T_{(\tau, \prec_\text{lex})}$ if and only if $\{1,2,3\}$ is a 2-simplex in $\tau$. 

We first show that if there is a 3-simplex $\{1, 2, 3, k\}$ in $T_{(\tau, \prec_\text{lex})}$ for some $k>3$, then $\{1,2,3\}$ and $\{1,2,k\}$ are 2-simplices in $\tau$. To see this, we observe that there must exist 2-simplices $\tau_1\prec_\text{lex} \tau_2 \prec_\text{lex} \tau_3 \prec_\text{lex} \tau_4$ in $\tau$ such that \[(\tau_1\cap\tau_3)\cup (\tau_2\cap\tau_4) = \{1,2,3,k\}.\] The 3-simplex $\{1,2,3,k\}$ could arise as a join of 1-simplices in the following three combinations,
\[\{1,2\}\cup \{3,k\} \; , \; \{1,3\}\cup \{2,k\} \; , \; \{1,k\}\cup\{2,3\}. \]
However, a straight-forward case analysis (left to the reader) shows that the only combination which can be consistent with the ordering $\prec_\text{lex}$ is if we have
\[\tau_1\cap \tau_3 = \{1,3\} \; \text{ and } \; \tau_2\cap\tau_4 = \{2,k\},\]
and consequently we must have
\[\tau_1 = \{1,2,3\}  , \;
\tau_2 = \{1,2,k\}  , \;\tau_3 = \{1,3,l\}  , \;\tau_4 = \{2,k,m\},  \]
for some $k,l>3$ and $m \geq 3$. This proves the claim that $\{1,2,3\}$ and $\{1,2,k\}$ are both 2-simplices in $\tau$. Note that this also implies that if $\{1,2,3\}$ is not a 2-simplex in $\tau$, then there are no 3-simplices in $T_{(\tau,\prec)}$ which contain the triple $\{1,2,3\}$.

Conversely, for fixed $k>3$, we claim that if $\{1,2,3\}$ and $\{1,2,k\}$ are 2-simplices in $\tau$, then the 3-simplex $\{1,2,3,k\}$ appears an odd number of times in $T_{(\tau, \prec_\text{lex})}$. Since $\tau$ is a 2-cycle, we have an even number of  2-simplices which contain the pair $\{1,3\}$, and an odd number of them have the form $\{1,3,l\}$ with $l>3$. For the same reason, we have an odd number of 2-simplices which have the form $\{2,k, m\}$ with $m\geq 3$, and the claim follows. 

Finally, suppose $\{1,2,3\}$ is a 2-simplex in $\tau$. Since $\tau$ is a 2-cycle, we have an odd number of 2-simplices of the form $\{1,2,k\}$ with $k>3$. Therefore we find an odd number of distinct integers $k>3$ such that the 3-simplex $\{1,2,3,k\}$ is in $T_{(\tau, \prec_\text{lex})}$.
\end{proof}

\subsection{Proof of Theorem \ref{3homol}} Suppose there exists a counterexample, that is, a connected cover $\mathcal F$ of a planar graph $G$ such that $\tilde{H}_3(N(\mathcal{F}))\neq 0$. By Lemma \ref{2connected} we may assume that $\mathcal{F}$ is a 2-connected cover of $G$. Clearly we must have $|\mathcal{F}|\geq 5$, and we will assume that $|\mathcal{F}|$ is minimal in the following sense: if $\mathcal{H}$ is a connected cover of a planar graph and $|\mathcal{H}|<|\mathcal{F}|$, then $\tilde{H}_3(N(\mathcal{H})) = 0$. (Note that $|\mathcal{F}|$ does not change when we apply Lemma \ref{2connected}.)

Fix an embedding of $G$ in $\mathbb{S}^2$ and consider a member $G_x\in \mathcal{F}$. If every face of $G_x$ is a face of $G$, then $G_x = G$. But this would imply that $N(\mathcal{F})$ is a cone, which is contractible. We may therefore assume that $G_x$ has a face $f$ which is not a face in $G$. Now we fill the face $f$ in $G_x$ to obtain a new family $\mathcal{F}'$. It follows by definition that $|\mathcal{F}'| = |\mathcal{F}|$, and by Lemma \ref{filling is connected} that $\mathcal{F}'$ is a 2-connected cover of $G$. Our goal is to show that $\tilde{H}_3(N(\mathcal{F}'))\neq 0$, and consequently, that $\mathcal{F}'$ is also a minimal counterexample.

This will complete the proof, because we may continue the process of filling faces of $G_x$ while maintaining the property of being a minimal counterexample. At each step $G_x$ will have more vertices than before, and the process terminates when there are no more faces of $G_x$ to fill, in which case $G_x =G$. But then the nerve is contractible, giving us a contradiction.

\medskip

Now to show that $\tilde{H}_3(N(\mathcal{F}'))\neq 0$, let $U$ be the subset of vertices of $G$ which are in the interior of $f$ and recall that there is an inclusion $N(\mathcal{F}) \subset N(\mathcal{F}')$. Observe that a simplex $\sigma'$ is in $N(\mathcal{F}')\setminus N(\mathcal{F})$ if and only if $\sigma' = \sigma \cup \{x\}$ and the vertices of $G_\sigma$ are contained in $U$. Consequently, we have $\sigma \in N(\mathcal{F}) -x$. We define the subset 
\[\Delta  = \{ \sigma \in N(\FF) :   V(G_\sigma)\subset U\},\]
which is a subset of simplices of $N(\FF)-x$. Closing $\Delta$ downwards we obtain a simplicial complex 
\[L = \{\tau : \tau \subset \sigma \in \Delta\} \subset N(\FF)-x,\]
and if we let $K = N(\FF)$, then we have $N(\FF') = K_{(x,L)}$.

\medskip

Now suppose $\tilde{H}_3(K_{(x,L)}) = 0$.  Then there exists an $(x,L)$-critical 3-cycle $\gamma$, and by Lemma \ref{critical} we may assume $\gamma = \beta + [x, \partial\beta]$
where 
\begin{enumerate}
\item\label{one} $\beta = \sum\sigma_i \in C_3(L) \subset C_3(K-x)$, 
\item\label{three} $\sigma_i \cup \{x\} \in K_{(x,L)}\setminus K$ for all $i$, and
\item\label{two} $[x,\partial\beta] \in C_3(K)$.
\end{enumerate}

Let us interpret these conditions in terms of the connected cover. Conditions \eqref{one} and \eqref{three} imply that for every 3-simplex $\sigma_i$ in $\beta$, the vertices of $G_{\sigma_i}$ are contained in $U$. Condition \eqref{two} implies that for every 2-simplex $\tau_i$ in $\partial \beta$, we have $G_{\tau_i} \cap G_x\neq\emptyset$, and therefore the graph $G_{\tau_i}$ must contain vertices in $U$ and vertices in $G_x$. 

Since $G_x$ is 2-connected, the face $f$ is bounded by a cycle $S$. Therefore, for any 2-simplex $\tau_i$ in $\partial\beta$, the connected graph $G_{\tau_i}$ contains at least one vertex of $S$. So for every $\tau_i$ in $\partial \beta$ we may assign a vertex $v_{\tau_i}\in S\cap G_{\tau_i}$ (chosen arbitrarily). Note that distinct 2-simplices $\tau_{i}$ and $\tau_{j}$ could be assigned to the same vertex of $S$.

By orienting the cycle $S$, we may define a linear ordering $\prec$ on the 2-simplices of $\partial \beta$ as follows. Choose an arbitrary ``starting vertex'' and an orientation of the cycle $S$ to obtain a linear ordering $\prec_S$ of the vertices in $S$. The linear ordering $\prec_S$ induces a partial ordering $\prec'$ of the 2-simplices of $\partial \beta$ by setting $\tau_{i} \prec' \tau_{j}$ if and only if $v_{\tau_i} \prec_S v_{\tau_j}$. Finally, let $\prec$ be an arbitrary linear extension of $\prec'$.

Since $\partial \beta$ is a simplicial 2-cycle with a linear ordering $\prec$ of its simplices, we get a simplicial 3-chain $T = T_{(\partial \beta, \prec)}$ as defined in section \ref{boundary identity}. We claim that 
\begin{enumerate} \setcounter{enumi}{3}
\item\label{four} $T\in C_3(K-x)$, and
\item\label{five} $[x,T] \in C_4(K)$.
\end{enumerate}
Since $\partial\beta\in Z_2(K-x)$ it follows by the definition of $T_{(\partial\beta, \prec)}$ that none of the simplices in $T$ contain the vertex $x$. Therefore statement \eqref{four} follows as a consequence of statement \eqref{five}.

To see why statement \eqref{five} holds, consider a 3-simplex $\sigma = \{a,b,c,d\} \in T$. We want to show that $\sigma\cup \{x\}\in K$, or equivalently, $G_\sigma\cap G_x \neq \emptyset$.
The contribution of $\sigma$ to $T$ comes from an ordered sequence of 2-simplices in $\partial\beta$,
\[\tau_1 \prec \tau_2 \prec \tau_3 \prec \tau_4,\]
where $\{a,b\}\in \tau_1\cap \tau_3$ and $\{c,d\}\in \tau_2\cap \tau_4$. Suppose some $\tau_i$ and $\tau_{i+1}$ are assigned to the same vertex of $S$, that is, $v_{\tau_i} = v_{\tau_{i+1}}$ (where subindices are taken modulo 4). Since every vertex of $S$ belongs to $G_x$ and $\sigma \subset \tau_i\cup \tau_{i+1}$, this would imply that $G_\sigma\cap G_x \neq\emptyset$, in which case we are done. So we may assume that the $\tau_i$ are assigned to distinct vertices of $S$. 

The vertices $v_{\tau_1}$ and $v_{\tau_3}$ are both contained in $G_{\{a,b,x\}}$. Since $\mathcal{F}$ is a connected cover there exists a path $\pi_{ab}$ connecting $v_{\tau_1}$ to $v_{\tau_3}$ which is contained in the graph $G_{\{a,b,x\}}$. Similarly, there is a path $\pi_{cd}$ connecting $v_{\tau_2}$ to $v_{\tau_4}$ contained in the graph $G_{\{c,d,x\}}$. Note that the paths $\pi_{ab}$ and $\pi_{cd}$ do not enter the interior of the face $f$. So due to the ordering of their endpoints along the cycle $S$ and the planarity of $G$, the paths $\pi_{ab}$ and $\pi_{cd}$ must cross at a vertex which belongs to $G_\sigma \cap G_x$. This completes the proof of statement \eqref{five}. (See Figure \ref{fig:crossings}.)

\begin{figure}[ht]
\centering
\begin{tikzpicture}
\begin{scope}
\coordinate (v0) at (190:2.1cm);
\coordinate (v1) at (160:2.5cm);
\coordinate (v2) at (140:2.3cm);
\coordinate (v3) at (110:2cm);
\coordinate (v4) at (80:1.8cm);
\coordinate (v5) at (50:1.6cm);
\coordinate (v6) at (20:1.7cm);
\coordinate (v7) at (-10:1.5cm);
\coordinate (v8) at (-50:1.3cm);
\coordinate (v9) at (-80:1.1cm);
\coordinate (v10) at (-140:1.4cm);
\draw[white] (-4.5,0)--(-3,1);
\draw[violet] (v2) ..controls (-2,3) and (2,3) .. (v5);
\draw[teal] (v3) ..controls (2,3) and (3.4,-1) .. (v8);
\fill (v1) circle (2pt);
\fill (v2) circle (2pt);
\fill (v3) circle (2pt);
\fill (v4) circle (2pt);
\fill (v5) circle (2pt);
\fill (v6) circle (2pt);
\fill (v7) circle (2pt);
\fill (v8) circle (2pt);
\fill (v9) circle (2pt);
\draw[dotted] (v1)--(v0)--(v10)--(v9);
\draw (v1)--(v2)--(v3)--(v4)--(v5)--(v6)--(v7)--(v8)--(v9);
\node at (-.4,.5) {\footnotesize $f$};
\node[below] at ($(v2)+(0.2,0)$) {\footnotesize $v_{\tau_1}$};
\node[below] at ($(v3)+(0.1,0)$) {\footnotesize $v_{\tau_2}$};
\node[left] at ($(v5)+(0,-0.2)$) {\footnotesize $v_{\tau_3}$};
\node[above] at ($(v8)+(-0.2,0)$) {\footnotesize $v_{\tau_4}$};
\node at (-1.8,2.4) {\footnotesize $\pi_{ab}$};
\node at (2.1,1.4) {\footnotesize $\pi_{cd}$};
\end{scope}
\end{tikzpicture}
\caption{The paths $\pi_{ab}$ and $\pi_{cd}$ cross at a vertex in $(\cap_{k\in \sigma}G_k)\cap G_x$. \label{fig:crossings}}
\end{figure}
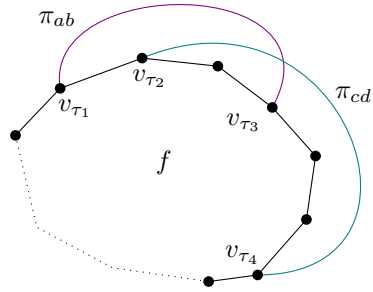

\medskip

Now we can finish the proof of Theorem \ref{3homol}. By
Lemma \ref{isoperm} we have $\partial T = \partial \beta$, so by statements \eqref{one} and \eqref{four} we have 
\[\beta + T \in Z_3(K-x).\]
Note that $K-x = N(\FF\setminus\{G_x\})$, so by the minimality of $|\FF|$ we have 
\[\beta = T+ \varphi\]
for some $\varphi \in B_3(K-x) \subset B_3(K)$. 

By the ``product rule'' we have
\[\partial [x,T] = T + [x,\partial T] = T + [x,\partial \beta],\]
which gives us
\[\gamma = \beta + [x,\partial\beta] = \partial[x,T] + \varphi \in B_3(K),\]
by statement \eqref{five}.
This contradicts the assumption that $\gamma$ is an $(x,L)$-critical 3-cycle and shows that $\tilde{H}_3(N(\mathcal{F}')) \neq 0$. \qed

\section{Helly-type theorems for connected covers} \label{helly-thms}

\subsection{Helly's theorem}
Let $\FF = \{S_1, \dots, S_n\}$ be a family of sets. The {\bf \em Helly number} of $\FF$ is defined as the greatest integer $m = h(\FF)$ for which there exists a subfamily ${\mathcal G} \subset \FF$ of cardinality $m$ such that every proper subfamily of $\mathcal{G}$ is intersecting and $\bigcap_{S\in \mathcal G} S = \emptyset$. Helly's classical theorem \cite{helly23} asserts that finite families of convex sets in $\mathbb{R}^d$ have Helly number at most $d+1$.

When we are dealing with connected covers in graphs, the following theorem shows that the role of the dimension in Helly's theorem is replaced by the size of the largest complete minor. 

\begin{thm}\label{hely}
Let $\FF$ be a connected cover in a graph $G$. If $G$ does not have a $K_r$ minor, then $h(\FF) < r$.
\end{thm}

\begin{proof}
Suppose $h(\FF) = m$. This means that there exists a subfamily $\mathcal{G} = \{ G_1, \dots, G_m\}$ of $\FF$ such that $N(\mathcal{G})$ is the boundary of the $(m-1)$-dimensional simplex. So, for every $(m-2)$-simplex $\sigma_1, \dots, \sigma_m$ in $N(\mathcal{G})$, where $\sigma_i = [m]\setminus \{i\}$, the graphs $G_{\sigma_1}, \dots, G_{\sigma_m}$ are non-empty, connected, and pairwise disjoint. 

For every $(m-3)$-simplex $\sigma_i\cap \sigma_j$ in $N(\mathcal{G})$, we choose a shortest path $\pi_{i,j}$ contained in $G_{\sigma_i \cap \sigma_j}$ which connects a vertex in $G_{\sigma_i}$ to a vertex in $G_{\sigma_j}$. Since $G_{\sigma_i}$ and $G_{\sigma_j}$ are pairwise disjoint, it follows that the path $\pi_{i,j}$ has at least one edge. By the minimality of $\pi_{i,j}$ it follows that only the endpoints of the path are contained in $G_{\sigma_i}$ and $G_{\sigma_j}$, respectively, while any interior vertex of the path (if there are any) is contained in $G_{\sigma_i\cap \sigma_j} \setminus V(G_{\sigma_i}\cup G_{\sigma_j})$. Note that for $j\neq k$, the paths $\pi_{i,j}$ and $\pi_{i,k}$ can have at most one vertex in common which must necessarily be a common endpoint contained in $G_{\sigma_i}$. Moreover, if
 $i,j,k$ and $l$ are all distinct, then the paths $\pi_{i,j}$ and $\pi_{k,l}$ are pairwise disjoint, otherwise there is a vertex contained in every member of $\mathcal{G}$. This shows that $K_m$ is a minor in $G$. \end{proof}


\subsection{A \texorpdfstring{$\bm{(p,q)}$}{(p,q)} theorem for connected covers} 
A far-reaching generalization of Helly's theorem is the celebrated $(p,q)$ {\em theorem}, conjectured by Hadwiger and Debrunner \cite{hd57}, and proved by Alon and Kleitman \cite{pqthm}. This theorem asserts that finite families of convex sets in $\mathbb{R}^d$ which satisfy the $(p,q)$ property\footnote{Recall from the introduction that the $(p,q)$ property means that among any $p$ members some $q$ of them intersect.} (for some $p\geq q \geq d+1$) have piercing number bounded by a constant which depends only on $p$, $q$, and $d$. There are numerous generalizations of the $(p,q)$ theorem \cite{boundingpiercing, transversal, lattice, hell, pablo} and the problem of obtaining good estimates on the piercing number of families of convex sets satisfying the $(p,q)$ property is a major open problem in discrete geometry. 

In \cite{transversal}, Alon et al.~generalized the $(p,q)$ theorem to abstract set-systems (hypergraphs) and showed that any set-system which satisfies an appropriate {\em fractional Helly property} will automatically also satisfy the assertion of the $(p,q)$ theorem as well. 

Let $\FF$ be an arbitrary family of sets, and let $\beta(\alpha) : (0,1] \to (0,1]$ be a function. We say that $\FF$ satisfies the {\em fractional Helly property for $k$-tuples with density function $\beta$} if for any finite subfamily $\mathcal{G}\subset \FF$ where at least $\alpha\binom{|\mathcal{G}|}{k}$ of the $k$-membered subfamilies of $\mathcal{G}$ are intersecting, there is an intersecting subfamily of $\mathcal{G}$ containing at least $\beta(\alpha) |\mathcal{G}|$ members. 

The fractional Helly theorem, due to Katchalski and Liu \cite{KatLiu} asserts that families of convex sets in $\mathbb{R}^d$ satisfy the fractional Helly property for $(d+1)$-tuples for some density function $\beta(\alpha)$ where $\beta(\alpha) \to 1$ as $\alpha \to 1$. Later it was shown by Eckhoff \cite{eck-upper}, and independently by Kalai \cite{upperbound}, that in this case the optimal density function is $\beta(\alpha) = 1-(1-\alpha)^{1/(d+1)}$. 

\medskip

Let $\FF$ be a family of sets, and let $\FF^{\cap} = \{\textstyle{\bigcap_{S\in \mathcal{G}}} S : \mathcal{G}\subset \FF\}$ denote the family of all intersections of sets in $\FF$. The main result of \cite{transversal} asserts that if $\FF^{\cap}$ satisfies the fractional Helly property for $k$-tuples, then $\FF$ also satisfies a $(p,q)$ theorem for all $p\geq q \geq k$. (The statement we give here is in the union of Theorems 8 and 9 of \cite{transversal}.)

\begin{thm}[Alon et al. \cite{transversal}] \label{abs pq}
Given integers $p\geq q\geq k\geq 2$ and a function $\beta:(0,1] \to (0,1]$, there exists an integer $h = h(p,q,k, \beta)$ such that the following holds.
Let $\FF$ be a family of sets and suppose $\FF^{\cap}$ satisfies the fractional Helly property for $k$-tuples with density function $\beta$. If $\FF$ has the $(p,q)$ property, then the piercing number of $\FF$ is at most $h$.  
\end{thm}

\medskip

We now turn to the proof of Theorem \ref{connected pq}. Let $\mathcal{C}$ be a finite family of open connected sets in the plane such that any non-empty intersection of members in $\mathcal{C}$ is connected. We may assume that each member in $\mathcal{C}$ is bounded, because one can find a large open disk $D$ such that $N(\mathcal{C})$ and $N(\mathcal{C}')$ are isomorphic where $\mathcal{C}' = \{C_i \cap D: C_i \in \mathcal{C}\}$. (Here we are using the assumption that $\mathcal{C}$ is finite.)
Since open connected sets in the plane are path-connected, we can approximate $\mathcal{C}$ by a connected cover $\FF$ in a planar graph in the sense that $N(\mathcal{C})$ and $N(\FF)$ are isomorphic.
One way to do this is by taking a triangulation of the disk $D$ which we view as a planar graph $G=(V,E)$, and associate each member $C_i$ in $\mathcal{C}$ with the induced subgraph $G_i = G[C_i \cap V]$. This gives us a family $\FF$ of induced subgraphs of $G$. If the edge lengths of the triangulation are sufficiently small, then $\FF$ will be a connected cover in $G$ and $N(\FF)$ will be isomorphic to $N(\mathcal{C})$. 
 

In order to apply Theorem \ref{abs pq} we need to show that $\FF$ satisfies the fractional Helly property for triples. (Note that if $\FF$ is a connected cover in a graph $G$, then by definition $\FF^{\cap}$ is also a connected cover in $G$.) In particular we have the following fractional Helly theorem.

\begin{thm} \label{fractional helly}
For every $\alpha\in (0,1]$ there exists a $\beta = \beta(\alpha)>0$ such that the following holds. Let $\FF$ be a connected cover in a $K_5$ minor-free graph. If at least $\alpha\binom{|F|}{3}$ of the triples in $\FF$ are intersecting, then there is an intersecting subfamily in $\FF$ of size at least $\beta|\FF|$. 
\end{thm}

Applying Theorem \ref{abs pq} we obtain the following corollary which contains Theorem \ref{connected pq} as a special case.

\begin{cor}\label{pq K5}
For any positive integers $p\geq q \geq 3$ there exists an integer $t = t(p,q)$ such that the following holds.
Let $\FF$ be a connected cover in a $K_5$ minor-free graph. If $\FF$ has the $(p,q)$ property, then piercing number of $\FF$ is at most $t$. 
\end{cor}

\subsection{Proof of the fractional Helly theorem}
A simplicial complex $K$ is called {\bf \em $\bm d$-Leray} if $\tilde{H}_i(L) = 0$ for all $i\geq d$ and every induced subcomplex $L\subset K$. This notion arises naturally in the study of Helly-type theorems due to the fact that 
the nerve of a finite family of convex sets in $\mathbb{R}^d$ is $d$-Leray. 
 
Let $f_j(K)$ denote the number of $j$-dimensional simplices in $K$. A deep generalization of the fractional Helly theorem for convex sets is the ``upper bound theorem'' for $d$-Leray complexes due to Kalai \cite{upperbound, shifting} (see also \cite[Theorem 13]{transversal}). This theorem gives precise upper bounds on $f_j(K)$ for $d\leq j < d+r$ in terms of $f_0(K)$ provided $f_{d+r}(K) = 0$. We will need the following consequence of Kalai's upper bound theorem.\footnote{Kalai's upper bound theorem for $d$-Leray complexes is usually stated for homology with coefficients in $\mathbb{Q}$, but also holds when using coefficients in $\mathbb{Z}_2$ by a standard application of the universal coefficient theorem.}

\begin{thm}[Kalai \cite{upperbound, shifting}] \label{d-leray bound}
Let $K$ be a $d$-Leray complex with $f_0(K) = n$. If $f_d(K)>\binom{n}{d+1} - \binom{n-r}{d+1}$, then  $f_{d+r}(K)>0$.
\end{thm}

As an immediate consequence of Theorem \ref{d-leray bound} we get that if $f_d(K)>\alpha \binom{n}{d+1}$, then $f_{\lfloor \beta n \rfloor}(K) >0$ where $\beta = 1 - (1-\alpha)^{1/(d+1)}$. Now let $\FF$ be a connected cover in a $K_5$ minor-free graph $G$. By Theorem \ref{main-thm}, $N(\FF)$ is a $3$-Leray complex, which implies the following.

\begin{cor} \label{strong fractional} Let $\FF$ be a connected cover in a $K_5$-minor-free graph $G$. If the number of intersecting subfamilies of $\FF$ of size four is at least $\alpha\binom{|\FF|}{4}$, then there is an intersecting subfamily of $\FF$ of size at least $\beta |\FF|$ where $\beta = 1-(1-\alpha)^{1/4}$. 
\end{cor}

This is almost what we need to prove Theorem \ref{connected pq}, and by Theorem \ref{abs pq} it implies a $(p,q)$ theorem for all $p\geq q \geq 4$. To complete the proof of Theorem \ref{fractional helly} we need to show $\FF$ satisfies the fractional Helly property for triples. This requires some additional combinatorial arguments. 

\medskip

The first tool we need is a well-known theorem due to Erd{\H o}s and Simonovits \cite[Corollary 2]{ErdSim}. Given integers $r\geq 2$ and $m_1, \dots, m_r\geq 1$ let $K^{(r)}_{(m_1, \dots, m_r)}$ denote the complete multipartite $r$-uniform hypergraph with vertex classes of size $m_1, \dots, m_r$. That is, let the vertex set of $K^{(r)}_{(m_1, \dots, m_r)}$ consist of disjoint sets $V_1, \dots, V_r$ where $|V_i| = m_i$ and let the edges be all $r$-tuples which contain exactly one vertex from each $V_i$. The Erd{\H o}s--Simonovits theorem asserts that any dense $r$-uniform hypergraph contains {\em many} distinct copies of $K^{(r)}_{(m_1, \dots, m_r)}$. We state here only the case that is needed in the proof of Theorem \ref{fractional helly}.

\begin{thm}[Erd{\H o}s--Simonovits]\label{erdos Simonovits}
For any $c_1>0$ there exists a $c_2>0$ such that if a $3$-uniform hypergraph $\mathcal{H}$ has at least $c_1n^3$ edges, then $\mathcal{H}$ contains at least $c_2n^{20}$ copies of $K^{(3)}_{(5,5,10)}$. 
\end{thm}

The second tool we need is a ``weak colorful Helly theorem'' for connected covers. Suppose $\FF$ is a connected cover in $G$ and suppose each member of $\FF$ has been colored with one of three distinct colors such that each color is used sufficiently many times. The assertion of the ``weak colorful Helly theorem'' is that if every colorful triple is intersecting (a colorful triple is one in which each color appears), then either some four members of $\FF$ intersect, or $G$ contains a large complete minor. We only state the version needed for the proof of Theorem \ref{fractional helly}, but more general versions can be obtained by the same proof method. (See the discussion in section \ref{final}.) 

\begin{lem} \label{weakcol}
Let $\FF$ be a connected cover in a graph $G$ with $|\FF| = 20$. Suppose there is a partition $\FF = \mathcal{A} \cup \mathcal{B} \cup \mathcal{C}$ where $|\mathcal{A}| = |\mathcal{B}| = 5$ and $|\mathcal{C}| = 10$, and that $A_i\cap B_j\cap C_k\neq \emptyset$ for every transversal $A_i\in \mathcal{A}$, $B_j\in \mathcal{B}$, $C_k\in \mathcal{C}$. If there are no four members of $\FF$ that intersect, then $G$ has a $K_5$ minor.
\end{lem}

\begin{proof} Let $\mathcal{A} = \{A_1, \dots, A_5\}$, $\mathcal{B} = \{B_1, \dots, B_5\}$, and $\mathcal{C} = \{C_{1,2}, C_{1,3}, \dots, C_{4,5}\}$. Assuming that no four members in $\mathcal{F}$ intersect, we will construct a $K_5$ minor. For every $1\leq i \leq 5$, let $G_i$ be the subgraph defined as 
\[G_i = (A_i\cap B_i) \cup \left(\textstyle{\bigcup_{j>i}}(A_i\cap C_{i,j}) \right).\] 
We first note that $G_i$ is connected for every $i$. To see this, observe that $G_i$ is a union of connected subgraphs, and that $(A_i\cap B_i) \cap (A_i\cap C_{i,j}) = A_i\cap B_i \cap C_{i,j} \neq \emptyset$ for every $j$. Therefore $G_i$ is connected. Next, we observe that if $G_i \cap G_j \neq\emptyset$ for some $i<j$, then some four members of $\mathcal{F}$ intersect, so we may  assume that the subgraphs $G_1, \dots, G_5$ are pairwise disjoint. 

\medskip

The next step is to modify each subgraph $G_i$ to obtain a new subgraph $G_i'$. This will be done such that the $G_i'$ remain pairwise disjoint, and such that for every $i<j$ there is an edge connecting a vertex in $G_i'$ to a vertex in $G_j'$, resulting in a $K_5$ minor. The modification consists of attaching certain paths to each $G_i$. 

For fixed $i < j$ we define a path $\pi_{i,j}$ as follows. By the assumption that $A_k\cap B_j \cap C_{i,j}\neq \emptyset$ for every $k$, it follows that there is a path contained in $B_j\cap C_{i,j}$ which connects a vertex in $A_i\cap C_{i,j}$ to a vertex in $A_j\cap B_j$. Therefore there exists a path contained in $(B_j\cap C_{i,j})\setminus A_i$ which has one endpoint in $A_j\cap B_j \subset G_j$ and the other endpoint adjacent to a vertex in $A_i\cap C_{i,j}\subset G_i$. Let $\pi_{i,j}$ be such a path. (Note that it is possible that $\pi_{i,j}$ consists of a single vertex.)

We now make two observations concerning the paths $\pi_{i,j}$. The first observation is that, if $\pi_{i,j} \cap \pi_{k,l} \neq \emptyset$, then we must have $j=l$.  To see this, note that  \[\pi_{i,j}\cap \pi_{k,l}\subset B_j\cap C_{i,j} \cap B_l\cap C_{k,l},\] and so if $j\neq l$, then there are four members of $\mathcal{F}$ that intersect. The second observation is that $G_k\cap \pi_{i,j}\neq \emptyset$ if and only if $k=j$. In one direction, it follows by construction that $G_j\cap \pi_{i,j}\neq \emptyset$. For the other direction, suppose $k\neq j$. Since $G_i\subset A_i$ it follows by construction that $G_i\cap \pi_{i,j} = \emptyset$, so we may assume $k\notin\{i,j\}$. But in this case neither $B_j$ nor $C_{i,j}$ are involved in the definition of $G_k$, and again if $G_k\cap \pi_{i,j}\neq \emptyset$, then some four members of $\FF$ intersect.

To complete the construction, we  define
\[G_j' = G_j \cup \left( \textstyle{\bigcup_{i<j}\pi_{i,j}} \right).\]
By the construction of the paths $\pi_{i,j}$ and the observations above it follows that the subgraphs $G_1', \dots, G_5'$ are connected, pairwise disjoint, and for every $i<j$ there is an edge connecting a vertex in $G_i'$ to a vertex in $G_j'$.
\end{proof}

\begin{proof}[Proof of Theorem \ref{fractional helly}]
Let $\FF$ be a connected cover in a $K_5$-minor free graph and let $n = |\FF|$.
We will show that if there are at least $\alpha\binom{n}{3}$ intersecting triples in $\FF$, then the number of intersecting subfamilies of $\FF$ of size four is at least $\alpha' \binom{n}{4}$ where $\alpha'>0$ depends only on $\alpha$. By Corollary \ref{strong fractional} there is an intersecting subfamily of $\FF$ of size at least $\beta n$ where $\beta = 1-(1-\alpha')^{1/4}$

Let $\mathcal{H}$ denote the 3-uniform hypergraph whose vertices and edges correspond to the vertices and 2-simplices of $N(\FF)$. Applying the Erd{\H o}s--Simonovits theorem (Theorem \ref{erdos Simonovits}) to $\mathcal{H}$ we find that there are at least $c_1n^{20}$ copies of $K^{(3)}_{(5,5,10)}$ in $\mathcal{H}$ where $c_1>0$ depends only on $\alpha$. By Lemma \ref{weakcol}, for every copy of $K^{(3)}_{(5,5,10)}$ in $\mathcal{H}$, among the corresponding members of $\FF$ there is an intersecting subfamily of size four. Each such intersecting subfamily is contained in at most $c_2n^{16}$ distinct copies of $K^{(3)}_{(5,5,10)}$ in $\mathcal{H}$ for some absolute $c_2>0$ (which also takes into account the constant number copies of $K^{(3)}_{(5,5,10)}$ on a fixed subset of 20 vertices). Thus there is an $\alpha'>0$, depending only on $c_1$ and $c_2$, such that $\FF$ contains at least $\alpha' \binom{n}{4}$ intersecting subfamilies of size four. 
\end{proof}

\section{Concluding remarks} \label{final}

\subsection{A general fractional Helly theorem} The ``weak colorful Helly theorem'' stated as Lemma \ref{weakcol} can be generalized to force arbitrarily large complete minors using an arbitrary number of color classes. The proof method is the same and yields the following. 

\begin{lem}\label{gen weakcol}
Let $\mathcal{F}$ be a connected cover in a graph $G$ with $|F|=2r+\binom{r}{2}+(N-3)$ for some $N \geq 3$.
Suppose there is a partition $\mathcal{F}=\mathcal{A}_1\cup\cdots\cup\mathcal{A}_N$ where $|\mathcal{A}_{N-2}|=|\mathcal{A}_{N-1}|=r$, $|\mathcal{A}_N|=\binom{r}{2}$ and $|\mathcal{A}_i|=1$ for $i < N-2$ such that $A_1\cap\cdots\cap A_N \neq \emptyset$ for every transversal $A_i \in \mathcal{A}_i$.
If there are no $N+1$ members of $\mathcal{F}$ that intersect, then $G$ has a $K_{r}$ minor.
\end{lem}

By repeated applications of the Erd{\H o}s--Simonovits theorem in combination with Lemma \ref{gen weakcol} as in the proof of Theorem \ref{fractional helly}, we can show that connected covers in graphs of bounded homological dimension satisfy the fractional Helly property for triples. 

\begin{prop}\label{gen frac hel}
For every $\alpha\in(0,1]$ and positive integer $n$, there exists a $\beta = \beta(\alpha, n)>0$ such that the following holds.
Let $\mathcal{F}$ be a connected cover in a graph $G$ with $\gamma(G)\leq n$. If at least $\alpha \binom{|\FF|}{3}$ of the triples in $\FF$ are intersecting, then there is an intersecting subfamily in $\FF$ of size at least $\beta|\FF|$.
\end{prop}

\subsection{A \texorpdfstring{$\bm{(p,q)}$}{(p,q)} conjecture} We conjecture that Theorem~\ref{connected pq} admits a generalization to arbitrary surfaces. For instance, if Conjecture \ref{minor-nerve} holds, then Proposition \ref{gen frac hel} together with Theorem \ref{abs pq} will imply the following.

\begin{conj}\label{pq surf}
For any integers $p\geq q\geq 3$ and surface $S$,  there exists an integer $C= C(p,q, S)$ such that the following holds. Let $\mathcal{F}$ be a finite family of open connected subsets of $S$ satisfying the condition that the intersection of any members of $\mathcal{F}$ is empty or connected.
If $\mathcal{F}$ has the $(p,q)$ property, then the piercing number of $\mathcal{F}$ is at most $C$.
\end{conj}

\subsection{Colorful Helly theorems}
The colorful Helly theorem discovered by Lov{\' a}sz, and independently by B{\'a}r{\'a}ny \cite{colcar}, is another classical generalization of Helly's theorem. Kalai and Meshulam \cite{topcol} showed that this result can be extended to arbitrary $d$-Leray complexes, and as a consequence we get the following. 

\begin{cor}
Let $\FF$ be a connected cover in a $K_5$ minor-free graph, and suppose there is a partition $\FF = \FF_1 \cup \FF_2 \cup \FF_3 \cup \FF_4$ into non-empty parts. If every colorful $4$-tuple is intersecting, then $\mathcal{F}_i$ is an intersecting family for some $1 \leq i \leq 4$.
\end{cor}

By the same proof method as in \cite{Kim17}, we also get the following colorful variant of the fractional Helly theorem.

\begin{thm}
For every $\alpha\in(0,1]$, there exists $\beta = \beta(\alpha)\in(0,1]$ tending to $1$ as $\alpha$ tends to $1$ such that the following holds. Let $\FF$ be a connected cover in a $K_5$-minor-free graph, and suppose there is a partition $\FF = \FF_1 \cup \FF_2 \cup \FF_3 \cup \FF_4$ into non-empty parts. If at least $\alpha \prod_{i=1}^4 |\FF_i|$ of the colorful 4-tuples are intersecting, then for some $1\leq i \leq 4$, $\mathcal{F}_i$ contains an intersecting subfamily of size $\beta |\mathcal{F}_i|$.
\end{thm}
We also note that if Conjecture \ref{minor-nerve} holds, then the statements above will generalize to connected covers in arbitrary $K_r$ minor-free graphs.

\subsection{Extension to higher dimensions} It is straight-forward to extend the notions introduced in this paper to higher-dimensional simplicial complexes, but so far we have no results in this direction. As a first step, we conjecture the following generalization of Theorem \ref{planar}.

\begin{conj}
There exists an integer $d$ such that the following holds.
Let $K$ be a 2-dimensional simplicial complex, and let $\FF =\{K_1, \dots, K_n\}$ be subcomplexes of $K$ where $K_\sigma = \bigcap_{i\in \sigma}K_i$ is simply connected for every $\sigma\in N(\FF)$. If $K$ is embeddable in $\mathbb{R}^4$, then $N(\FF)$ is $d$-Leray.
\end{conj}

\section*{Acknowledgment}
The authors are grateful to an anonymous referee for helpful comments which improved the exposition of our results.


\begin{thebibliography}{100}

\bibitem{epsilon} N.~Alon, I.~B{\' a}r{\' a}ny, Z.~F{\" u}redi, D.~J.~Kleitman, Point selections and weak $\varepsilon$-nets for convex hulls, {\em Combin. Probab. Comput.} 1 (1992), 189--200.

\bibitem{boundingpiercing} N.~Alon and G.~Kalai, Bounding the piercing number, {\em Discrete Comput. Geom.} 13 (1995), 245--256.

\bibitem{transversal} N.~Alon, G.~Kalai, J.~Matou{\v s}ek, R.~Meshulam, Transversal numbers for hypergraphs arising in geometry, {\em Adv. Appl. Math.} 29 (2002), 79--101.

\bibitem{pqthm} N.~Alon and D.~J.~Kleitman, Piercing convex sets and the Hadwiger--Debrunner $(p,q)$-problem, {\em Adv. Math.}  96 (1992), 103--112.

\bibitem{colcar} I.~B{\' a}r{\' a}ny, A generalization of Carath{\' e}odory's theorem, {\em Discrete Math.} 40 (1982), 141--152.

\bibitem{lattice} I.~B{\' a}r{\' a}ny and J.~Matou{\v s}ek, A fractional Helly theorem for convex lattice sets, {\em Adv. Math.} 174 (2003), 227--235.

\bibitem{bjorner-survey} A.~Bj{\" o}rner, Topological methods, in: {\em Handbook of combinatorics} (R.~L.~Graham, M.~Gr{\"o}tschel, L.~Lov{\'a}sz, editors)  Vol. 2, {North-Holland, Elsevier Sci. B. V., Amsterdam} (1995), 1819--1872. 

\bibitem{bjorner} A.~Bj{\" o}rner, Nerves, fibers and homotopy groups, {\em J. Combin. Theory Ser. A} 102 (2003), 88--93.
 
\bibitem{borsuk} K.~Borsuk, On the imbedding of systems of compacta in simplicial complexes, {\em Fund. Math.} 35 (1948), 217--234.

\bibitem{xavi-adv} E.~Colin~de~Verdi{\` e}re, G.~Ginot, X.~Goaoc,  Helly numbers of acyclic families, {\em Adv. Math.} 253 (2014), 163--193. 

\bibitem{diestel} R.~Diestel, Graph theory, GTM 173, 5th edition, Springer Verlag, 2016

\bibitem{eckhoff} J.~Eckhoff, Helly, Radon, and Carath{\' e}odory type theorems, in: {\em Handbook of convex geometry} (P.~M.~Gruber, J.~M.~Wills, editors) Part. A, North-Holland, Elsevier Sci. B. V., Amsterdam (1993), 389--448. 

\bibitem{eck-upper} J.~Eckhoff, 
An upper-bound theorem for families of convex sets, 
{\em Geom. Dedicata} 19 (1985), 217--227. 

\bibitem{edmonds} J.~Edmonds, M.~Laurent, and A.~Schrijver, A minor-monotone graph parameter based on oriented matroids, {\em Discrete Math.} 165/166 (1997), 219--226. 

\bibitem{ErdSim} P.~Erd{\H o}s and M.~Simonovits,
Supersaturated graphs and hypergraphs,
{\em Combinatorica} 3 (1983), 181--192.

\bibitem{fary} I.~F{\'a}ry, On straight-line representation of planar graphs, {\em Acta Sci. Math. (Szeged)} 11 (1948), 229--233. 

\bibitem{xbetti} X.~Goaoc, P.~Pat{\'a}k, Z.~Pat{\'a}kov{\'a}, M.~Tancer, and U.~Wagner, Bounding Helly numbers via Betti numbers,  in: {\em A Journey Through Discrete Mathematics} (M.~Loebl, J.~Ne{\v s}et{\v r}il, R.~Thomas, editors), Springer, Cham (2017), 407--447.

\bibitem{govc} D.~Govc and P.~Skraba, An approximate nerve theorem, To appear in Found Comput Math.

\bibitem{hd57}
H.~Hadwiger and H.~Debrunner, {\"U}ber eine {Variante} zum {Helly}'schen {Satz}, \emph{Arch. Math.} 8 (1957), 309--313.

\bibitem{hell} S.~Hell, On a topological fractional Helly theorem, arXiv:math/0506399, 2005.

\bibitem{helly23}
E.~Helly, {\"U}ber {Mengen} konvexer {K{\"o}rper} mit gemeinschaftlichen
  {Punkten}, {\em Jahresber. Deutsch. Math.-Verein.} 32 (1923), 175--176.

\bibitem{wenger} A.~Holmsen and R.~Wenger, Helly-type theorems and geometric transversals, in: {\em Handbook of discrete and computational geometry}  (J.~E.~Goodman, J.~O'Rourke, C.~D.~T{\'o}th, editors), 
CRC Press Ser. Discrete Math. Appl., CRC, Boca Raton (2017), 91--123. 

\bibitem{vdh} H.~van~der~Holst, M.~Laurent, and A.~Schrijver, On a minor-monotone graph invariant, {\em J. Combin. Theory Ser. B} 65 (1995), 291--304.

\bibitem{upperbound} G.~Kalai, Intersection patterns of convex sets, {\em Israel J. Math.} 48 (1984), 161--174.

\bibitem{shifting} G.~Kalai, Algebraic shifting, in {\em Computational commutative algebra and combinatorics (Osaka, 1999)}, 121--163, Adv. Stud. Pure Math. 33, Math. Soc. Japan, Tokyo, 2002. 

\bibitem{topcol} G.~Kalai and R.~Meshulam, A topological colorful Helly theorem,
{\em Adv. Math.} 191 (2005), 305--311.

\bibitem{topamenta} 
G.~Kalai and R.~Meshulam, 
Leray numbers of projections and a topological Helly-type theorem, {\em J. Topol.} 1 (2008), 551--556.

\bibitem{KatLiu} M.~Katchalski and A.~Liu, A problem of geometry in $\mathbb{R}^n$, {\em Proc. Amer. Math. Soc.} 75 (1979), 284--288. 

\bibitem{Kim17} M.~Kim, A note on the colorful fractional Helly theorem, {\em Discrete Math.} 340 (2017), 3167--3170.

\bibitem{geza} D.~J.~Kleitman, A.~Gy{\' a}rf{\' a}s, and G.~T{\' o}th, Convex sets in the plane with three of every four meeting, 
{\em Combinatorica} 21 (2001), 221--232. 

\bibitem{meshulam} R.~Meshulam, The Clique Complex and Hypergraph Matching, {\em Combinatorica.} 21 (2001), 89--94.

\bibitem{mohar} B.~Mohar, What is $\dots$ a graph minor, {\em Notices Amer. Math. Soc.} 53 (2006), 338--339. 

\bibitem{montejano} L.~Montejano, A variation on the homological nerve theorem, {\em Topology Appl.} 225 (2017), 139--144.

\bibitem{graph minor theorem} 
N.~Robertson and P.~D.~Seymour,
Graph Minors. XX. Wagner's conjecture,
Journal of Combinatorial Theory Ser B, 92 (2004), 325--357

\bibitem{schr} A.~Schrijver, Minor-monotone graph invariants, in: {\em Surveys in combinatorics, 1997} (R.~A.~Bailey, editor),
London Math. Soc. Lecture Note Ser. 241, Cambridge Univ. Press, Cambridge (1997), 163--196. 

\bibitem{pablo} P.~Sober{\'o}n, Helly-type theorems for the diameter, {\em Bull. London Math. Soc.} 48 (2016), 577--588.

\bibitem{tancer} M.~Tancer, Intersection patterns of convex sets via simplicial complexes, a survey, in: {\em Thirty Essays on Geometric Graph Theory} (J. Pach, editor), 
Springer New York (2013), 521-540.

\bibitem{wagner} K.~Wagner, {\" U}ber eine Eigenschaft der ebenen Komplexe, {\em Math. Ann.} 114 (1937), 570--590. 


\end{thebibliography}
\end{document}